\documentclass[a4,11pt]{article}

\usepackage{caption}
\usepackage{color}
\usepackage{todonotes}
\usepackage{amsmath,amssymb,amsthm}
\RequirePackage[colorlinks,citecolor=blue,urlcolor=blue,linkcolor=blue]{hyperref}
\hypersetup{
colorlinks = true,
citecolor=blue,
urlcolor=blue,
linkcolor=blue,
pdfauthor = {Remco van der Hofstad, Mark Holmes,  Alexey Kuznetsov and Wioletta Ruszel},
pdftitle = {Strongly reinforced P\'olya urns with graph-based competition},
pdfkeywords = {reinforcement model, P\'olya urn,  stochastic approximation algorithm, stable equilibria}, 
pdfpagemode = UseNone
}
\usepackage{todonotes}
\usepackage{bbm}
\usepackage[title,titletoc]{appendix}
\usepackage{dsfont}
\usepackage{graphicx}

\newtheorem{THM}{Theorem}[section]
\newtheorem{PRP}[THM]{Proposition}
\newtheorem{LEM}[THM]{Lemma}
\newtheorem{CON}[THM]{Conjecture}
\newtheorem{EXA}[THM]{Example}

\newtheorem{COR}[THM]{Corollary}
\newtheorem{DEF}[THM]{Definition}
\newtheorem{COND}[THM]{Condition}
\newcommand{\R}{\mathbb{R}}
\renewcommand{\P}{\mathbb{P}}

\newcommand{\Z}{\mathbb{Z}}
\newcommand{\N}{\mathbb{N}}
\newcommand{\C}{{\bf C}}
\renewcommand{\c}{\mathbb{C}}
\newcommand\1{\mathds{1}}
\newcommand{\indic}[1]{\1_{\{#1\}}}
\newcommand{\sss}{\scriptscriptstyle}

\newcommand{\mc}[1]{\mathcal{#1}}

\def\ra{\rightarrow}

\newcommand\Qed{$\qed$ \medskip}
\def\re{\textnormal {Re}}
\def\im{\textnormal {Im}}

\newcommand{\blank}[1]{}
\newcommand{\FIGS}[1]{#1}

    \def\d{{\textnormal d}}
 \def\beq{\begin{eqnarray}}
    \def\eeq{\end{eqnarray}}
    \def\beqq{\begin{eqnarray*}}
    \def\eeqq{\end{eqnarray*}}

\usepackage{color}

\newcommand{\nn}{\nonumber}

\newcommand{\Aa}{\mc{A}_{\alpha}}
\newcommand{\Ea}{\mc{E}_{\alpha}}
\newcommand{\Sa}{\mc{S}_{\alpha}}

\newcommand{\adj}{\mathrm{adj}}

\newcommand{\e}{{\mathrm e}}

\title{Strongly reinforced P\'olya urns with graph-based competition}
\author{
{Remco van der Hofstad
\footnote{Department of Mathematics and Computer Science, Eindhoven University of Technology, the Netherlands. \newline
E-mail: rhofstad@win.tue.nl}} \; ,
{Mark Holmes
\footnote{Department of Statistics, University of Auckland, New Zealand. Email:  mholmes@stat.auckland.ac.nz}} \;,
{Alexey Kuznetsov
\footnote{Department of Mathematics and Statistics, 
York University, Canada. Email: kuznetsov@mathstat.yorku.ca}}\;, 
and 
{Wioletta Ruszel
\footnote{Department of Applied Mathematics, Delft University of Technology, the Netherlands. Email: W.M.Ruszel@tudelft.nl}} 
}
\begin{document}
\maketitle
\abstract{We introduce a class of reinforcement models where, at each time step $t$, one first chooses a random subset $A_t$ of colours (independent of the past) from $n$ colours of balls, and then chooses a colour $i$ from this subset with probability proportional to the number of balls of colour $i$ in the urn raised to the power $\alpha>1$.  We consider stability of equilibria for such models and establish the existence of phase transitions in a number of examples, including when the colours are the edges of a graph, a context which is a toy model for the formation and reinforcement of neural connections.}
{\vskip 0.15cm}
{\vskip 0.15cm}

\tableofcontents
%\listoftodos

\newpage

\section{Introduction}
\label{sec:intro}
Random processes with reinforcement have been studied mathematically since at least the early 1900s, and have connections to applied problems such as the design of clinical trials, and the formation of networks such as neural networks, the Internet, and social networks.  One of the most simple and elegant of these models is known as {\em P\'olya's urn}:  starting with one black and one red ball in an urn, select a ball at random from the urn, replace it, and add another of the same colour.  The proportion $X_t$ of black balls in the urn after $t$ balls have been added is a bounded Martingale, and has a discrete uniform distribution for each $t$, whence there is a random variable $X\sim U[0,1]$ such that $\P(X_t\ra X)=1$.  Various generalisations of this model have been studied in the last hundred years or so, see e.g.~\cite{Mah09,Pem07}.  In recent times, reinforced random walks and preferential attachment models continue to be studied extensively.

One direction of generalisation of P\'olya's urn is to modify this selection probability (the probability of selecting a ball of a given colour).  Fix $W:\mathbb{N}\ra (0,\infty)$, and if $N_t^{\sss (i)}$ is the number of balls of colour $i$ in the urn at time $t$, then at time $t+$ we select a ball of colour $i$ from the urn with probability $W(N_t^{\sss (i)})/\sum_j W(N_t^{\sss (j)})$.   In P\'olya's urn, there are two colours and $W(x)=x$.  A beautiful construction due to Rubin \cite{Dav90} shows that if $\sum_{x=1}^{\infty}W(x)^{-1}<\infty$ (sometimes called the \emph{strong} reinforcement regime) then only one colour is chosen infinitely often.  Otherwise each colour is chosen infinitely often, and if $W$ grows sufficiently slowly (e.g.~$W(x)=x^{\alpha}$ for some $\alpha\in (0,1)$) then the proportions of each colour are equal in the limit.  

A further direction of generalisation involves having multiple interacting urns, where colours may be present in more than one urn and where multiple balls may be added to one or more urns depending on what colour is selected.  See for example the PhD Thesis (and related papers) of Launay \cite{LaunayThesis,Launay1,Launay2}, and recent work of Launay and Limic \cite{LL12}, and Benaim and coauthors \cite{BRS12,BBCL12}.  In such settings colours may not be competing with each other on every iteration of the process, and Rubin's construction need not apply.

\subsection{Our models}
\label{sec:model}
Consider the following simplistic model for the reinforcement of neural connections in the brain:  A signal enters the brain at some (randomly) chosen neuron and is transmitted to a (random) single neighbouring neuron with probability depending on the relative efficiency of the synapses connecting the neurons, and in doing so the efficiency of the utilized synapse is improved/reinforced.  We are interested in the structures and relative efficiency of neural networks that can arise from repeating this process a very large number of times, in a strong reinforcement regime.

With this motivation, we consider a large class of ``interacting urn"-type models that we have not found in the literature.  Suppose that we have $n$ colours of balls.  Let $N_{t}^{\sss (i)}$ be the number of balls of colour $i$ in our ``urn" at time $t\in \Z_+$, and assume that $N_{0}^{\sss (i)}=1$ for each $i\in [n]=\{1,2,\dots, n\}$.  The process $\vec{N}_t=(N_{t}^{\sss (i)}:i\in [n])$ evolves as follows.  At time $t\in \N$ we choose a subset $A_t \subset [n]$ independently of $\mc{F}_{t-1}=\sigma\{A_{s},(N_{s}^{\sss (i)}:i\in [n])\}_{0\le s\le t-1}$, according to some law.  We then select a colour $i$ from the balls of colours in $A_t$ according to their current weights in the urn, i.e., given $A_t$, we select a ball of colour $i\in A_t$ with probability
\begin{align}
\frac{W(N_{t-1}^{\sss (i)})}{\sum_{j \in A_t}W(N_{t-1}^{\sss (j)})}\label{ball_select}
\end{align}
then we replace that ball and add another of the same colour, so that $N_t^{\sss (j)}=N_{t-1}^{\sss (j)}+\indic{j=i}$.  For a fixed $n$, the law of such a model is then completely specified by the function $W$ and the law of $A_1$, so we will refer to any such model as a $(W,A)$-Reinforcement Model, or simply a {\em WARM}.

%%%%%%%%%%%%%%%%%%%%%%%%%%%%%%%%%%%%%
%\subsection{Reinforcement schemes}
In \cite{HK_Exp_urn} we consider the case 
$W(x)=\e^{\gamma x}$ for some fixed $\gamma >0$.  In this paper our results will be for reinforcement functions $W:\N\ra (0,\infty)$ of the following kind: 
%satisfying one of the following conditions:
\medskip

{\noindent \bf Condition $(\alpha)$:} $W(x)=x^{\alpha}$ for some fixed $\alpha>1$.

\medskip
%{\noindent \bf Condition $(\gamma)$:} 
%\begin{itemize}
%\item[$(\alpha)$] $W(x)=x^{\alpha}$ for some fixed $\alpha>0$, or
%\item[$(\gamma)$] $W(x)=\e^{\gamma x}$ for some fixed $\gamma >0$.
%%\item[(a)] $W$ is {\em non-decreasing}
%%\item[(b)] $\sum_{x\in \Z_+}W(x)^{-1}<\infty$ {\em summable}
%%\item[(c)] For any $c>1$, $W(cx)/W(x)\ra \infty$ as $x \ra \infty$ {\em (super-polynomial)}.
%\end{itemize}

We will also assume the following condition. 
\begin{COND}[Subset selection]
\label{cond:iid}
The subsets $(A_t)_{t \ge 0}$ are i.i.d.~with $p_{\varnothing}=0$, where $p_A\equiv \P(A_t=A)$.
\end{COND} 

We are interested in the random vectors  $\vec{X}_t=\vec{N}_t/(t+n)$ of proportions of balls of each colour,
 and more precisely their limits as $t\ra \infty$.  
Any model with $p_{\varnothing}\in (0,1)$ can be considered as a random time change of a model with $p_{\varnothing}=0$, which does not affect the possible limits of $\vec{X}_t$.  Thus we lose nothing in assuming that $p_{\varnothing}=0$ in Condition \ref{cond:iid}.
%while improving our presentation.  
For models with plenty of symmetry in terms of the colour labellings, we may instead consider the ordered vector $[\vec{X}_t]$, having the same elements as $\vec{X}_t$, but listed in decreasing order. 
%Note that $W(x)=x^{\alpha}$ satisfy (a), as well as (b) when $\alpha>1$ and $W(x)=\e^{cx}$ for $c>0$ (which satisfies (a)-(c)).
Most of our examples satisfy the following symmetry property, which implies that $\P(|A_1|=m)=nm^{-1}a_mp_m$:
\begin{COND}[Symmetry]
\label{cond:symmetry}
There exist $(p_\ell)_{\ell=1}^n$ and $(a_\ell)_{\ell=1}^n$ such that for every $m\ge 0$,
\begin{itemize}
\item[(i)] $p_A\in \{0,p_m\}$ whenever $|A|=m$, and
\item[(ii)] $\#\{A\ni i:|A|=m, p_A=p_m\}=a_m$  for every $i\in [n]$.
\end{itemize}
\end{COND}
%Note that under this condition, $\P(|A_1|=m)=\frac{n}{m}a_mp_m$
%$p_{\varnothing}=p_0$ and $p_{[n]}=p_n$, and moreover 
%\[\P(|A_1|=m)=\sum_{\stackrel{A\sqsubset [n]:}{|A|=m}}p_m=\frac{1}{m}\sum_{i=1}^n\sum_{\stackrel{A\sqsubset [n]:}{A\ni i,|A|=m}}p_m=\frac{n}{m}a_mp_m.\]

Condition \ref{cond:symmetry} is somewhat unpalatable, so let us point out that many of the models considered in this paper satisfy the following stronger symmetry property, which implies that $\P(|A_1|=m)={n \choose m}p_m$, and also that (almost surely) at least $n-\underline{m}+1$ colours are chosen a positive proportion of the time, where $\underline{m}=\min\{m\ge 1:p_m>0\}$.
\begin{COND}[Strong symmetry]
\label{cond:strong_symmetry}
There exist $(p_i)_{i=1}^n$ such that $p_A=p_m$ whenever $|A|=m$.
%\begin{align}
%\label{symmetry_prop}
%\text{there exist }p_0,\dots,p_n \text{ such that }p_A=p_m\quad  \text{whenever} \quad |A|=m.
%\end{align}
\end{COND}
%If Condition \ref{cond:strong_symmetry} holds then 
%\[\P(|A_1|=m)={n \choose m}p_m.\]
%It is easy to see that Condition \ref{cond:strong_symmetry} also implies that 
We are primarily interested in the setting where the colours $[n]$ are the edges (synapses) of a connected graph $G$ (brain) with $n_v$ vertices (neurons).  In this setting we will assume the following.
%%The colours are the edges of the graph.  
%The set $A_t$ is chosen by choosing a random vertex $V_t$ from all $n_v$ vertices in the graph, and letting $A_t$ be the set of edges incident to $V_t$.  We will assume that $V_t$ is chosen uniformly at random as in the following condition.
%% $A_t=\{(V_t,y):y \in G\setminus \{V_t\}\}$.  
%%Let $H=H(G)$ denote the (random) set of edges chosen infinitely often.
\begin{COND}
\label{cond:Vuniform}
$V_t$ is chosen uniformly at random from the vertices of $G$, and $A_t$ is the set of edges incident to $V_t$.
% i.e.~$A_t=\{(V_t,y):y \in G\setminus \{V_t\}\}$.  
\end{COND}
WARMs where the law of $A_1$ corresponds to Condition \ref{cond:Vuniform} on some graph $G$ will be called {\em WARM graphs}.  When $G$ is specified the WARM will be called a {\em WARM (on) $G$}.

%%%%%%%%%%%%%%%%%%%%%%%%%%%%%%%%%%%%
\subsection{Examples}
\label{sec:examples}
We begin with two WARMs that are in general not WARM graphs.
%models that are (in general) non-graph-based. 
\begin{EXA}[Uniform, fixed $m$]
\label{exa:fixedm}
Fix $m \in [n]$ (the model becomes relatively trivial when $m=1$ or $m=n$) and choose $A_t$ with $|A_t|=m$ uniformly at random from $[n]$.  Then $|A_t|=m$ almost surely and $\P(A_t=A)=m!(n-m)!/n!$ when $|A|=m$.  This is the special case of Condition \ref{cond:strong_symmetry} with $p_r=0$ for all $r\ne m$.  At least $n-m+1$ colours are each chosen a positive proportion of the time.  
\end{EXA}

\begin{EXA}[Bernoulli$(p)$]
\label{exa:bernoulli}
Fix $p \in (0,1)$, and independently choose each colour to be in $A_t$ with probability $p$.  After a parameter change (due to $p_{\varnothing}=(1-p)^n>0$), this is the special case of Condition \ref{cond:strong_symmetry} with $p_m=p^m(1-p)^{n-m}(1-p_{\varnothing})^{-1}$ for all $m\ge 1$. All $n$ colours are chosen a positive proportion of the time.  
\end{EXA}
A natural extension of Example \ref{exa:bernoulli} would be to have a different $p$ for each colour.  Turning to WARM graphs (i.e., assuming Condition \ref{cond:Vuniform} hereafter), observe that the special case of Example \ref{exa:bernoulli} with $n=2$ and $p=1/2$ is the same as the WARM on the star-graph on 2 edges.
\begin{EXA}[WARM Star graph]
\label{exa:star_graph}
Let $G$ be the star-graph on $n_v=n+1$ vertices consisting of a central vertex connected by $n$ edges to $n$ leaves (vertices of degree 1).  
%The set $A_t$ is chosen by first choosing a random vertex 
Then the WARM on $G$ is the special case of Condition \ref{cond:strong_symmetry} with 
$p_1=p_n=1/(n+1)$ and $p_m=0$ otherwise.
\end{EXA}

%Let us now switch our attention to models (on graphs) that do not satisfy Condition \ref{cond:strong_symmetry}.   
In the next two examples, $G$ is regular with degree $d=d(n)$ (so $|A_t|=d$ almost surely),  so the WARM on $G$ satisfies Condition \ref{cond:symmetry} with $p_A=0$ if $|A|\ne d$, and with $p_{d}=1/n_v$ and $a_{d}=2$ since any one of the $n_v$ vertices is equally likely to be $V_t$ and every edge is incident to 2 vertices.  On the other hand there exist subsets of size $d$ that are chosen with probability 0 (so Condition \ref{cond:strong_symmetry} is not satisfied).

\begin{EXA}[WARM Cycle graph]
\label{exa:circle}
Let $G$ be the cycle graph with $n$ edges and $n$ vertices. Each vertex is of degree $d=2$.
\end{EXA}

\begin{EXA}[WARM Complete graph]
\label{exa:complete_graph}
Let $G$ be the complete graph on $n_v$ vertices, with $n=n_v(n_v-1)/2$ edges.  Each vertex is of degree $d=n_v-1$.
\end{EXA}
Note that Examples \ref{exa:complete_graph}, \ref{exa:circle}, and \ref{exa:fixedm} (with $m=2$) are all identical when $n=3$, and correspond to the WARM {\em triangle graph} which is studied extensively in Section \ref{sec:triangle}.  All of the above examples satisfy the symmetry property Condition \ref{cond:symmetry}.
%, which guarantees that $\vec{1}/n$ is an equilibrium.  Let us now give a simple example for which $\vec{1}/n$ is not an equilibrium (except when $n=2$).
Let us now give a simple example that does not satisfy Condition \ref{cond:symmetry}(ii).
\begin{EXA}[WARM Line/Path graph]
\label{exa:line}
Let $G$ be the line segment with $n$ edges (and $n+1$ vertices).  The two leaves have degree 1, while all interior vertices have degree 2.
\end{EXA}
%Note that Examples \ref{exa:line} and \ref{exa:star_graph} are identical when $n=2$, and (together with Example \ref{exa:fixedm} with $m=n-1$) are closely related to the model of Example \ref{exa:bernoulli} with $n=2$. 
 Star graphs and the line graph with $n=3$ are special cases of {\em whisker graphs} (which also fails to satisfy Condition \ref{cond:symmetry} in general) defined as follows. 
\begin{EXA}[WARM Whisker graph]
\label{exa:whisker}
A whisker graph is defined as a tree with a diameter less than or equal to three. If the diameter is equal to two, then we obtain a star graph. If the diameter is equal to three, then we have a graph consisting of a distinguished edge $e$ with $r\ge 1$ leaves incident to one endvertex of $e$ and $s=n-(r+1)\ge 1$ leaves incident to the other endvertex (i.e.~$G$ is constructed by connecting two star graphs by a single edge, $e$).  
\end{EXA} 
We believe that whisker graphs play a central role in the graph setting (see Conjecture \ref{con:whisker-forest}).

%When $m=2$, Example \ref{exa:fixedm} is similar to the model studied in \cite{BCL12}, however in \cite{BCL12} one is simultaneously choosing a ball from every possible $A_t$ (instead of a single uniformly chosen one).  
%%%%%%%%%%%%%%%%%%%%%%%%%%%%%%%%%
\subsection{Linearly stable equilibria}
\label{sec:stable}
For fixed $n$ and $\vec{v}\in \Delta_n\equiv\{\vec{u}\in \R^n:u_i\ge 0, \sum_{i=1}^n u_i=1\}$, let $F\colon \Delta_n\ra \R^n$ be defined (for a given WARM) by 
\begin{align}
F(\vec{v})_i=-v_i+\lim_{t \ra \infty}\sum_{A\ni i}p_A\frac{W(v_it)}{\sum_{j \in A}W(v_jt)}.\label{Fdef}
\end{align}

\begin{DEF}[Equilibrium]
\label{def:equilibrium}
For fixed $n$, a vector $\vec{v}\in \Delta_n$ is an {\em equilibrium distribution} for the WARM if $F(\vec{v})=\vec{0}$.
%\begin{align}
%v_i=&\lim_{t \ra \infty}\sum_{A\ni i}\frac{p_A}{1-p_{\varnothing}}\frac{W(v_it)}{\sum_{j \in %A}W(v_jt)}.\label{equil_general}
%\end{align}
We let $\mc{E}$ denote the set of equilibria for a given WARM, and write $\mc{E}_{\alpha}=\mc{E}$ when Condition $(\alpha)$ holds.
\end{DEF}
Intuitively this says that in the limit as $t \ra \infty$, the proportion of balls of colour $i$ in the urn is equal to the probability that the next selected ball is of colour $i$.  To see that the term inside the limit in \eqref{Fdef} sums to 1, observe that
\begin{align}
\sum_{i=1}^n\sum_{A\ni i}p_A\frac{W(v_it)}{\sum_{j \in A}W(v_jt)}=\sum_{A \ne \varnothing}\sum_{i \in A}p_A\frac{W(v_it)}{\sum_{j \in A}W(v_jt)}=\sum_{A \ne \varnothing}p_A=1.\nn
\end{align}
Note that when $W(x)=x^{\alpha}$, $F(\vec{v})=\vec{0}$  reduces to 
\begin{align}
v_i=&\sum_{A\ni i}p_A\cdot \frac{v_i^{\alpha}}{\sum_{j \in A}v_j^{\alpha}}.\label{equil_alpha}
\end{align}
%and when $W(x)=\e^{\gamma x}$, \eqref{equil_alpha} reduces to 
%\begin{align}
%v_i=\sum_{A \ni i}\frac{p_A}{1-p_{\varnothing}}\cdot \frac{\indic{v_i\ge v_j %\forall j \in A}}{\#\{j \in A:v_i=v_j\}}.\label{equil_gamma}
%\end{align}
%%where we interpret the quantity in the expectation as 0 if $|A|=0$.

 Let the partial derivatives of $F$ at $\vec{v}$ be denoted by $D_{i,k}=\partial F(\vec{v})_i/\partial v_k$, whenever these quantities exist, and let ${\bf D}(\vec{v})$ denote the matrix with $(i,k)$ entry $D_{i,k}$ evaluated at the point $\vec{v}$.  
\begin{DEF}[Stable equilibrium]
\label{stable}
An equilibrium distribution $\vec{v}$ (i.e.~satisfying \eqref{equil_alpha}) is a {\em linearly-stable equilibrium} if all eigenvalues of ${\bf D}(\vec{v})$ have negative real parts,  {\em linearly-unstable equilibrium} if some eigenvalue of ${\bf D}(\vec{v})$ has positive real part, and {\em critical} otherwise.  
 Let $\mc{S}$ denote the set of linearly-stable equilibria for a given WARM, and write $\mc{S}_{\alpha}=\mc{S}$ when Condition $(\alpha)$ holds.  
\end{DEF}
For a given WARM, let $\mc{A}$(=$\mc{A}_{\alpha}$ when Condition $(\alpha)$ holds) denote the (random, nonempty) set of accumulation points of the sequence $\vec{X}_t$.  The main reason that we are interested in linearly-stable equilibria is because of the following theorem (and conjecture) whose proof (see Appendix \ref{sec:appendix}) relies on Theorem \ref{thm:number_equi} below together with the general theory of the dynamical system approach to studying stochastic approximation algorithms, established by Bena\"\i m and coauthors.  See for example \cite[Proposition 3.5, Theorem 3.9,  Theorem 3.11]{BRS12}.  
\begin{THM}[Accumulation structure]
\label{thm:convergence}
Assume Conditions $(\alpha)$ and \ref{cond:iid}.    Then 
\begin{enumerate}
\item[(i)] almost surely $\Aa\subset \Ea$ and $\Aa$ is a connected subset of $\Delta_n$,
\item[(ii)] $\P(\vec{X}_t\rightarrow \vec{v})>0$ for every $\vec{v}\in \Sa$.
\end{enumerate}
\end{THM}
%\RvdH{What is a connected subset of $\Ea$? Explain!}

It follows from Theorem \ref{thm:convergence}(i) that if $|\Ea|<\infty$ then ($|\Aa|=1$ almost surely so) $\vec{X}_t$ converges almost surely.  Moreover, if $|\Ea|=1$ then $\vec{X}_t$ converges almost surely to this unique equilibrium.  We shall see that when $n=2$ and $\alpha=3$ in Example \ref{exa:star_graph} there is a unique equilibrium ($\Ea=\{(1/2,1/2)\}$, whence $\vec{X}_t$ almost surely converges to $(1/2,1/2)$) that is not linearly stable ($\Sa=\varnothing$).
%We believe that $\vec{X}$ almost surely converges to a random element of $\Sa$.
\begin{CON}[Convergence to equlibirum]
\label{con:non-convergence}
For any WARM with $\alpha>1$, there exists a random vector $\vec{X}=(X_1,\dots,X_n)$, supported on the set of linearly-stable and critical equilibria such that $\P(\vec{X}_t \ra \vec{X})=1$.
\end{CON}

%Then there exists an $r \in \mathbb{N}$ such that $\mathbb{P}(\lim_{t \rightarrow \infty} \vec{X}_t =v_l) >0$ for all $l=1,...,r$ and $\vec{X}_t$ converges to $\{ v_1,...,v_r \}$ almost %surely. 
%\end{THM}
%The fact that there are only a finite number of linearly stable equilibria was established in theorem \ref{thm:number_equi} together with the fact that the equilibria are isolated. As a consequence of proposition 3.5. in \cite{BRS12} we comclude that almost surely the proportions of colours converge to one of them. The last claim of the theorem was already proven in theorem 3.11 in \cite{BRS12}. 

%%%%%%%%%%%%%%%%%%%%%%%%%%%%%%%%%%%
\subsection{Main results}\label{sec:main results}
Our main results describe the set $\Sa$ of linearly-stable equilibria in various situations, and hence (assuming Conjecture \ref{con:non-convergence}) the possible limiting proportions of balls of each colour.  
%As we shall discuss later, the general theory of stochastic dynamical systems (see ????) ensures that $\vec{X}_t$ converges to some random vector $\vec{X}$ almost surely, where $\vec{X}$ is supported on the set of stable equilibria.
In particular we are interested in phase transitions in the set $\Sa$ (including whether each colour can be chosen equally often) as $\alpha>1$ varies.
%%%%%%%%%%%%%%%%%%%%%%%%%%%%%%%%%%%%%
%\subsection{Stability of equilibria}
Our first main result states that the number of linearly-stable equilibria is finite under Conditions $(\alpha)$ and \ref{cond:iid}.
\begin{THM}[Finite number of stable equilibria]
\label{thm:number_equi}
Assume Conditions ($\alpha$) and \ref{cond:iid}.  Then $\Sa$ is finite.
\end{THM}
We will see some examples where $\Sa=\varnothing$, but in many cases the existence of at least one linearly-stable equilibrium is given by the following results, when $\alpha>1$ is sufficiently small.
\begin{LEM}[$\vec{1}/n$ equilibrium]
\label{lem:symmetry_ones}
Assume Conditions \ref{cond:iid} and \ref{cond:symmetry}.  Then $\vec{1}/n\in \Ea$.
\end{LEM}
Note that $\vec{1}/n$ is not an equilibrium for Example \ref{exa:line}, which does not satisfy Condition \ref{cond:symmetry}.  The eigenvalues associated to $\vec{1}/n\in \Ea$ will be continuous functions of $\alpha$, giving rise to the transitions between the linear-stable and critical regions:

\begin{PRP}[Stability of $\vec{1}/n$]
\label{prp:all_equal_stable}
Assume Conditions \ref{cond:iid}, \ref{cond:strong_symmetry}, $(\alpha)$, and that $p_n<1$.  Then $\vec{1}/n\in \Sa$ if and only if 
\begin{align}
\alpha<\frac{1}{n^2\sum_{m=2}^n\frac{p_m}{m^2}{n-2 \choose m-2}}.\label{all_equal_stable}
\end{align}
Moreover, $\vec{1}/n$ is critical if and only if equality holds in \eqref{all_equal_stable}.
\end{PRP}

Here the right hand side is strictly larger than 1 when $p_n<1$, so under the assumptions of Proposition \ref{prp:all_equal_stable}, $\vec{1}/n\in \Sa$ for $\alpha>1$ but sufficiently close (depending on the model) to $1$, and $\vec{1}/n\notin \Sa$ for $\alpha$ sufficiently large.  In other words, all such models exhibit at least one phase transition. By applying Proposition \ref{prp:all_equal_stable} to various special cases one obtains the following:
\begin{COR}[Stability of $\vec{1}/n$ in Examples \ref{exa:fixedm}--\ref{exa:star_graph}]
\label{cor:all_equal_stable}
The equilibrium $\vec{1}/n$ is linearly-stable (critical when equality holds below) for:
%for the models of the following Examples if and only if $\alpha$ 
\begin{itemize}
\item[] Example \ref{exa:fixedm} if and only if
\[\alpha<\frac{m(n-1)}{n(m-1)};    \]
\item[] Example \ref{exa:bernoulli} if and only if
\[\alpha<\frac{1-(1-p)^n}{\sum_{m=2}^np^m(1-p)^{n-m}\frac{n^2}{m^2}{n-2 \choose m-2}};\]
\item[] Example \ref{exa:star_graph} if and only if $\alpha<n+1$.
\end{itemize}
\end{COR}

The remaining examples have rather different behaviour:
\begin{PRP}[Stability of $\vec{1}/n$ in Examples \ref{exa:circle}--\ref{exa:complete_graph}]
\label{prp:cycle_complete_all_equal_stable}
The equilibrium $\vec{1}/n$ is linearly-stable (critical when equality holds below) for:
\begin{itemize}
\item[] Example \ref{exa:circle} if and only if $n$ is odd and $\alpha<\cos\left(\tfrac{\pi}{2n} \right)^{-2}$;
\item[] Example \ref{exa:complete_graph} if and only if $n=3$ and $\alpha<4/3$.
\end{itemize}
\end{PRP} 
Note that in the graph setting, when $\vec{v}=\vec{1}/n$, the matrix of partial derivatives is related to the edge-adjacency matrix.  Typically $\vec{1}/n$ is not the \emph{only} equilibrium, and indeed we will see many more examples of linearly-stable equilibria for various models.  See for example Corollary \ref{cor:cond_alpha} in the case of Example \ref{exa:fixedm}.  If Condition \ref{cond:Vuniform} (and Condition \ref{cond:iid}) is satisfied then by the law of large numbers, any $\vec{v}\in \mc{E}$ must satisfy $v_i\le 2/n_v$ for each edge $i$, since each edge is incident to 2 vertices.  Similarly, for any $i$ incident to a leaf we have $v_i\ge 1/n_v$. 

The following result often allows one to find stable equilibria in large systems by finding stable equilibria in smaller systems.
\begin{LEM}[Stability reduction]
\label{lem:reduction}
Suppose that $p_A=\P(A_t=A)$ for each $A\subset [n]$, with $p_{\{n\}}=0$,
%.
%$\P(A_t=\{n\})=0$, and we d
%efine 
and define 
%$\{p'_A,A\subset \C_{n-1}\}$ by 
$p'_{A\setminus\{n\}}=p_A$.  Then 
$\vec{v}=(v_1,\dots,v_{n-1},0)$ is a (linearly-stable) equilibrium for the WARM $(p_A)_{A\subset [n]}$ if and only if $\vec{v}'=(v_1,\dots, v_{n-1})$ is a (linearly-stable) equilibrium for the WARM $(p'_{A'})_{A'\subset \C_{n-1}}$.
\end{LEM}
The most important consequence of Lemma \ref{lem:reduction} for us is in the graph setting.
%We have established the following.
%\begin{LEM}
%\label{lem:symmetry_ones}
%Assume Conditions \ref{cond:iid} and \ref{cond:symmetry}.  Then $\vec{1}/n$ is an equilibrium distribution.
%\end{LEM}
 Let $G$ be a graph with vertex set $V$ and edge set $E$, and let $n_v=|V|$. Let $G_1,\dots, G_k$ be connected subgraphs of $G$ with $E_j=\{e_1,\dots,e_{|E_j|}\}$ and $V_j$ denoting the edge set and vertex set respectively of $G_j$.  Let 
\begin{align}
\mathcal{G}_{\sss G}=\big\{{\bf G}=\{G_j\}_{j=1}^k:k\le n_v/2, \, |V_j|\ge 2 \text{ for each }j, \,V=\cup_{j=1}^kV_j, \, V_j\cap V_{j'}=\varnothing \text{ for all }j\ne j'\big\},
\end{align}
denote the $V$-spanning collections of non-trivial connected clusters of $G$, and let ${\bf E}=\cup_{j=1}^kE_j$.  Let $\mc{E}_{\sss G}$ and $\mc{S}_{\sss G}$ denote the equilibria and linearly stable equilibria for a WARM on $G$.
\begin{THM}[Subgraph stability reduction]
\label{thm:subgraphs}
Fix $G$, and let 
	\[{\bf G}=\{G_j\}_{j=1}^k\in \mathcal{G}_{\sss G}\qquad \text{and} \qquad \vec{v}=\big((v_e)_{e\in E_1},(v_e)_{e\in E_2},\dots,(v_e)_{e\in E_k},(0)_{e\in E\setminus {\bf E}}\big).\]
Then, for any WARM on $G$,
\begin{itemize}
\item[(1)] $\vec{v}\in \mc{E}_{\sss G}$ if and only if $\frac{|V|}{|V_j|}(v_e)_{e\in E_j}\in \mc{E}_{\sss G_j}$ for each $j=1,\dots, k$,
\item[(2)] $\vec{v}\in \mc{S}_{\sss G}$ if and only if $\frac{|V|}{|V_j|}(v_e)_{e\in E_j}\in \mc{S}_{\sss G_j}$ for each $j=1,\dots, k$.
\end{itemize}
\end{THM}

\begin{DEF}[$(G,\alpha)$-stable allocation]Given a graph $G$, $\alpha>1$ and ${\bf G}\in \mc{G}_{\sss G}$, we say that ${\bf G}$ admits a $(G,\alpha)$-stable allocation if there exists $\vec{v}$ with $v_{e'}>0$ for all $e'\in {\bf E}$ and $v_e=0$ for all $e\in E\setminus {\bf E}$ such that $\vec{v}\in (\Sa)_{\sss G}$ or $\vec{v}$ is critical.  
\end{DEF}
An element ${\bf G}$ of $\mc{G}_{\sss G}$ is said to be a whisker-forest (resp.~star-forest) if each component $G_j$ is a whisker (resp.~star) graph.
\begin{CON}[Whisker-forest conjecture]
\label{con:whisker-forest}
Let $G$ be any graph. There exists $\alpha_{\sss G}$ such that, for all $\alpha>\alpha_{\sss G}$,
\begin{itemize}
\item[(i)] any whisker-forest ${\bf G}$ on $G$ admits a $(G,\alpha)$-stable allocation;
\item[(ii)] for any $\vec{v}\in (\Sa)_{\sss G}$, there exists a whisker-forest ${\bf G}$ on $G$ such that: $v_e>0$ if and only if $e \in {\bf E}$.
\end{itemize}
\end{CON}
As in Conjecture \ref{con:whisker-forest}(ii), we believe that when $\alpha$ is large enough (depending on $G$), all stable equilibria live on whisker forests. What we have proved in this direction is the following result, which follows from an explicit characterisation (see Theorem \ref{thm:general_star_graph}) of $\mc{S}_{\alpha}$ for the WARM star graph:

\begin{THM}[$(G,\alpha)$-stable allocation for star-forests]
\label{thm:star_forest_admits}
For any graph $G$, and $\alpha>1$, any star-forest ${\bf G}$ on $G$ admits a $(G,\alpha)$-stable allocation.
\end{THM}
For large $\alpha$, this result can be extended to symmetric-whisker-forests, where each non-star component is symmetric (i.e.~$s=r$) due to the following result:
\begin{THM}[Symmetric whisker graphs]
\label{thm:symmetric_whisker}
For every symmetric ($s=r$) whisker graph $G$, there exists $\alpha_{\sss G}>1$ such that for all $\alpha>\alpha_{\sss G}$, $(\Sa)_{\sss G}$ is non-empty. Consequently,  any symmetric whisker-forest ${\bf G}$ on $G$ admits a $(G,\alpha)$-stable allocation, if $\alpha$ is sufficiently large (depending on ${\bf G}$).
\end{THM}

\subsection{Overview of the paper}
In Section \ref{sec:general_proofs} we prove Theorem \ref{thm:number_equi}, Lemma \ref{lem:symmetry_ones}, Proposition \ref{prp:all_equal_stable}, and Theorem \ref{thm:subgraphs}.  In Section \ref{sec:special cases} we obtain more detailed results on the existence of linearly-stable equilibria for various examples.    In Section \ref{sec:discussion} we discuss some open problems, in particular some related to Conjecture \ref{con:whisker-forest}.  The proof of Theorem \ref{thm:convergence} follows the standard approach and is relegated to Appendix \ref{sec:appendix}.

%%%%%%%%%%%%%%%%%%%%%%%%%%%%%%%%%%%

\section{Proofs of general results}
\label{sec:general_proofs}
In this section we prove the general results of Section \ref{sec:main results}, assuming Conditions \ref{cond:iid} and $(\alpha)$ throughout.  We opt for a proof of Theorem \ref{thm:subgraphs} instead of proving Lemma \ref{lem:reduction}.  
%The proof of Lemma \ref{lem:reduction} is standard (see e.g.~Corollary 3.8 in \cite{BRS12}), 
We therefore begin this section with the proof that  there are only finitely many stable equilibria.

\subsection{Proof of Theorem \ref{thm:number_equi}}
Fix $\alpha>1$.  
%By a time change and reparametrization $p_A'=p_A/(1-p_{\varnothing})$ we may assume that $p_{\varnothing}=0$.  
For $n=1$, the claim is trivial.  The proof proceeds via induction over $n$, assuming that the result holds for all $n'<n$.

Let $\vec{v}=(v_1,\dots,v_n)\in \Ea$ denote an equilibrium distribution, 
%and $\mathcal{P}$ the collection of all subsets of $[n]$. For all $A_j \sqsubset [n]$, $\vec{v}$ satisfies
so that
\begin{equation}
F(\vec{v}) = \vec{0}, \label{a1}
\end{equation} 
where
\begin{equation}
F_i(\vec{v}) = -v_i + \sum_{A \ni i} p_{A}\frac{v_i^{\alpha}}{\sum_{k \in A} v_k^{\alpha}},\quad\qquad  \text{ for } i\in[n].
\end{equation}
We assume $v_i \neq 0$ for all $i\in[n]$. If an equilibrium is linearly stable 
for the system of $n$ equations and there is some $I\ne \varnothing$ such that $v_i=0$ for all $i \in I$, then it is linearly stable for the system on $[n] \setminus I$ (see Lemma \ref{lem:reduction} or \cite[Corollary 3.8]{BRS12}).

Let $A\subset [n]$ be non-empty.  Since $\alpha > 1$, using H\"older's inequality we have
\begin{equation}
\sum_{k \in A} v_k^{\alpha} \geq \biggl ( \sum_{k \in A} v_k \biggr )^{\alpha} |A|^{1-\alpha}. \label{m1}
\end{equation}
By the law of large numbers, the set $A$ is chosen a proportion $p_{A}$ of the time (in the limit as $t\ra \infty$).  It follows that colours contained in $A$ are chosen at least $p_{A}$ proportion of the time, hence
\begin{equation}
\sum_{k\in A} v_k \geq p_{A}.\label{m2}
\end{equation} 
Equations \eqref{m1} and \eqref{m2} yield
\begin{equation}
\sum_{k \in A} v_k^{\alpha} \geq p^{\alpha}_{A} |A|^{1-\alpha}.\label{abound}
\end{equation}
Inserting this into \eqref{a1} we obtain
\begin{align*}
v_{i}^{1-\alpha}=\sum_{A\ni i}p_A\frac{1}{\sum_{k\in A}v_k^{\alpha}}\le \sum_{A\ni i}p_A p^{-\alpha}_{A} |A|^{\alpha-1}.
\end{align*}
Thus,
\begin{equation*}
v_i \geq \biggl (\sum_{A \ni i} (p_{A}/ |A|)^{1-\alpha} \biggr )^{1/(1-\alpha)}.
\end{equation*}
This shows that there exists a (model dependent) $\varepsilon > 0$ such that there is no $\vec{v}\in \Ea$ satisfying $0 < v_i < \varepsilon$ for some $i \in [n]$. It remains to prove that for any $\varepsilon>0$ there are only finitely many $\vec{v}\in \Sa$ satisfying $v_i \geq \varepsilon$ for all $i \in [n]$. 

Fix $\varepsilon>0$, and choose $\delta\in (0,\frac{\pi}{2\alpha})$ and define $H^n \subset \c^n$ to be the Cartesian product of $n$ copies of the open complex domain
\begin{equation*}
H:=\biggl\{ z \in \mathbb{C}: \frac{\varepsilon}{2} < |z| < 2, |\arg(z)| < \delta \biggr\}.
\end{equation*}
Since, for $z\in H$,
\begin{align*}
|\arg(z^{\alpha})|=\alpha|\arg(z)|<\alpha \delta<\pi/2,
\end{align*}
we see that $\re(z^{\alpha})>0$ for all $z \in H$.  Therefore for non-empty $A$, 
$\re\big[\sum_{k\in A} v_k^{\alpha}\big]>0$ for $\vec{v}\in H^n$, in particular, all functions
$\vec{v}\mapsto \sum\limits_{k\in A} v_k^{\alpha}$ are analytic and zero-free in $H^n$, which shows that the functions 
$$
\vec{v} \mapsto \frac{v_i^{\alpha}}{\sum\limits_{k\in A} v_k^{\alpha}}
$$
are also analytic in $H^n$, so finally we conclude that the functions $F_i(\vec{v})$ are analytic in $H^n$. 

Next, define the map
$F\colon H^n \mapsto \c^n$ by 
$
F(\vec{v})=(F_1(\vec{v}),F_2(\vec{v}),\dots,F_n(\vec{v}))
$ and the set 
\begin{align*}
{\mathcal H}:=\{\vec{v}\in H^n \colon F(\vec{v})=\vec{0}\; {\textnormal{ and }} {\textnormal{det}}\left[{\bf D}(\vec{v})\right]\ne 0\}.
\end{align*}
Clearly $\Sa\subset {\mathcal H}$. Our goal is to show that (i) ${\mathcal H}$ is a set of isolated points and (ii) it does not have accumulation points in the interior of the domain $H^n$. 

\vspace{0.25cm}
\noindent
To prove (i),
%The first part of this statement (that all elements of ${\mathcal A}$ are isolated points) is relatively  straightforward. Let us d
let $\vec{w} \in {\mathcal H}$. Since $F(\vec{w})=\vec{0}$ and ${\textnormal{det}}\left[{\bf D}(\vec{w})\right]\ne 0$, due to the Implicit Function Theorem (see \cite[Theorem 2, page 40]{Shabat}) there exists a biholomorphic map between some neighborhoods $U \ni \vec{w}$ and $V\ni \vec{0}$ (that is, a bijective holomorphic function whose inverse is also holomorphic). Since the map is bijective, there are no other solutions to the system $F_i(\vec{v})=0$, $i\in[n]$ in $U$, which shows that each element of ${\mathcal H}$ must be an isolated point.

%\vspace{0.25cm}
%\noindent
%(ii) The second part is somewhat less obvious. We now know that ${\mathcal A}$ consists of isolated points, and we want to show that this set has no accumulation points in the interior of the domain $\Omega$. 
To prove (ii), let us assume the converse, i.e., there exists a point $\vec{w} \in H^n$ which is an accumulation point of ${\mathcal H}$. Define
$$
{\mathcal Z}:=\{\vec{v} \in H^n\colon F(\vec{v})=\vec{0}\},
$$
so ${\mathcal Z}$ is an {\it analytic set} in the sense of \cite[Definition 1, page 129]{Shabat}, and clearly ${\mathcal H}\subseteq {\mathcal Z}$. According to \cite[Theorem 2.2, page 52]{Nishino}, there exists a neighborhood $\Delta \subset H^n$ of the point $\vec{w}$, such that 
 the analytic set $\Delta \cap {\mathcal Z}$ can be decomposed into a {\it finite} number of pure-dimensional analytic sets. Pure-dimensional means that the set has the same dimension at each point. One of these pure-dimensional analytic sets must have dimension zero (since we have
 assumed that $\vec{w}$ is an accumulation point for isolated points in ${\mathcal H}$, and isolated points are zero-dimensional). It is also clear that this zero-dimensional analytic set must have an accumulation point at $\vec{w}$. Now we use \cite[Theorem 6 on page 135]{Shabat}, which says that this is impossible: any zero-dimensional analytic set in $\Delta$ cannot have limit points inside $\Delta$. Therefore, we have arrived at a contradiction.

So far we have proved that the set ${\mathcal H}$ consists of isolated points and does not have accumulation points in the interior of $H^n$. 
Since the  set 
\beqq
B:=\{ \vec{v} \in \c^n \colon \im(v_i)=0, \;\;\; \varepsilon\le \re(v_i) \le 1  \} 
\eeqq 
is compact in $\c^n$, we conclude that the set $B\cap H^n$ is finite. Since stable equilibria are elements of $B \cap H^n$,  this shows that we can have only finitely many $\vec{v}\in \Sa$ satisfying $v_i>\varepsilon$ for each $i$. \hfill\Qed
%Considering what we have established just above of \eqref{proof_eqn5}, we conclude that there can exist only finitely many stable equilibria. 
\blank{
For all $z \in D$ the real part of $z^{\alpha}$ is strictly positive since $|\arg(z^{\alpha})| < \frac{\pi}{2}$ and hence is the real part of $\sum_{k \in A_j} v_k^{\alpha}$.  The functions $F(\vec{v})=(F_1(\vec{v}),...,F_n(\vec{v}))$ are analytic in 
$D_{\alpha}^n$ since sums, products and ratios (with non-zero denominator) of analytic functions are analytic. Let us define the set $\mathcal{S}$ including in particular all linearly stable equilibria
\begin{equation*}
\mathcal{S}:= \{ \vec{v} \in D^n_{\alpha}: F(\vec{v})=\vec{0} \text{ and }  det(Jac(\vec{v})) \neq 0 \}
\end{equation*}
where $Jac(\cdot)$ denotes the Jacobian matrix related to $(D_{i,k}(\vec{v}))_{i,k}$.
In the following we show that $\mathcal{S}$ consists of isolated points and it does not have accumulation points in the interior of $D^n_{\alpha}$. 
For the first part let $\vec{w} \in \mathcal{S}$, then $F(\vec{w})=0$ and $det(Jac(\vec{w})) \neq 0$, due to the implicit function theorem there exists a biholomorphic map between neighbourhoods $U \ni \vec{w}$ and $V \ni \vec{0}$.
Since the map is bijective the solution to $F(\vec{w})$ is unique in $U$, hence each point in $\mathcal{S}$ must be an isolated point. In the remainder we show that it has no accumulation points in the interior of the domain. Let us assume there exists an accumulation point $w \in D^n_{\alpha}$ in $\mathcal{S}$ and define $\mathcal{S}_0 \supset \mathcal{S}$ as
\begin{equation*}
\mathcal{S}_0:= \{ \vec{v} \in D^n_{\alpha}: F(\vec{v})=\vec{0} \}
\end{equation*}
 then $\mathcal{S}_0$ is an \textit{analytic set} (page 129 \cite{N01}). According to theorem 2.2 in \cite{N01} there exists a neighbourhood $\mathcal{N}_{\alpha} \subset D_{\alpha}^n$ of $\vec{w}$ such that the analytic set $\mathcal{N}_{\alpha}\cap \mathcal{S}_0$ can be decomposed into a finite number of analytic sets which have the same dimension at each points. Since we assumed that $\vec{w}$ is an accumulation point, one of the sets has to have dimension zero and also an accumulation point at $\vec{w}$. But any zero-dimensional analytic set in $\mathcal{N}_{\alpha}$ cannot have limit points in the interior of $D^n_{\alpha}$ (theorem 6 on page 135 in \cite{S92}), so we arrive at a contradiction. 
We define
\begin{equation*}
\mathcal{B}_{\alpha}:=\{ \vec{v} \in \mathcal{C}^n:\mathfrak{I}(v_i)=0;  \varepsilon(\alpha) \leq \mathfrak{R}(v_i) \leq 1\}
\end{equation*}
 it is compact in $\mathbb{C}^n$. Since $\mathcal{S}$ consists of isolated points and has  no accumulation points inside $D^n_{\alpha}$ and $\mathcal{B}_{\alpha}$ is compact, we conclude $\mathcal{B}_{\alpha} \cap D^n_{\alpha}$ is finite. It follows that there can be only finitely many stable equilibria satisfying $v_i \geq \varepsilon(\alpha)$ for all $i=1,...,n$.
\Qed 
}

\subsection{Proof of Lemma \ref{lem:symmetry_ones} and Proposition \ref{prp:all_equal_stable}}

\noindent {\em Proof of Lemma \ref{lem:symmetry_ones}.}  
Assume that Condition \ref{cond:symmetry} holds.  Then, for $\vec{v}=\vec{1}/n$, the right hand side of \eqref{equil_alpha} becomes
\begin{align}
&\sum_{A\ni i}\frac{p_A}{|A|}=\sum_{m=1}^{n}\sum_{\stackrel{A\ni i:}{|A|=m}}\frac{p_A}{m}=\sum_{m=1}^{n}\frac{a_mp_m}{m},
\label{equil_equal}
\end{align}
which does not depend on $i\in [n]$.  Since these quantities sum to 1, it follows that the right hand side of \eqref{equil_alpha} is equal to $1/n$ for each $i$, which proves that $\vec{1}/n$ is an equilibrium (i.e., Lemma \ref{lem:symmetry_ones}).\hfill\Qed

Recall that the adjugate matrix $\adj{\bf A}$ of a square matrix ${\bf A}$ is given by $\adj{\bf A}={\bf C}^{\mathrm{T}}$, i.e., the transpose of the cofactor matrix ${\bf C}$ of ${\bf A}$.
Recall that if ${\bf A}$ is a diagonal matrix with entries $A_{ii}$, then its cofactor matrix is  a diagonal matrix ${\bf C}$ with $C_{ii}=\prod_{j\ne i}A_{jj}$, and its adjugate matrix is a diagonal matrix $\adj{\bf A}={\bf C}^{\mathrm{T}}={\bf C}$.
In order to prove Proposition \ref{prp:all_equal_stable}, we will use the following modification of the Matrix Determinant Lemma, which we have not found in the literature (although we expect that it is well known).
 \begin{LEM}[Modified Matrix Determinant Lemma]
\label{lem:MMDL}
If ${\bf R}\in \R^{n\times n}$ and $\vec{y},\vec{w}\in \R^n$ are column vectors then
\begin{align}
\det({\bf R}+\vec{y}\vec{w}^{\mathrm{T}})=\det({\bf R})+\vec{w}^{\mathrm{T}}\adj({\bf R})\vec{y}.
\label{MMDL}
\end{align}
\end{LEM}
\proof If ${\bf R}$ is invertible, then the matrix determinant lemma gives
\begin{align*}
\det({\bf R}+\vec{y}\vec{w}^{\mathrm{T}})=(1+\vec{w}^{\mathrm{T}}{\bf R}^{-1}\vec{y})\det({\bf R})=\det({\bf R})+\vec{w}^{\mathrm{T}}{\bf R}^{-1}\det({\bf R})\vec{y}=\det({\bf R})+\vec{w}^{\mathrm{T}}\adj({\bf R})\vec{y}.\nonumber
\end{align*}
If ${\bf R}$ is not invertible then ${\bf R}$ has some eigenvalues that are zero (and possibly some non-zero) and there exists some $\varepsilon_0$ (corresponding to the smallest magnitude-non-zero eigenvalue) such that no $\varepsilon\in (0,\varepsilon_0)$ is an eigenvalue for ${\bf R}$, i.e.,
$\det({\bf R} - \varepsilon {\bf I})\ne 0$ for all such $\varepsilon$.  Therefore, 
%$0$ is not an eigenvalue of ${\bf R} - \varepsilon {\bf I}$  
${\bf R} - \varepsilon {\bf I}$ is invertible for any such $\varepsilon$.  It follows that for all $\varepsilon\in (0,\varepsilon_0)$ 
\begin{align}
\det({\bf R}- \varepsilon {\bf I}+\vec{y}\vec{w}^{\mathrm{T}})=\det({\bf R}- \varepsilon {\bf I})+\vec{w}^{\mathrm{T}}\adj({\bf R}- \varepsilon {\bf I})\vec{y}.\label{epsilon}
\end{align}
We obtain the desired conclusion by taking the limit as $\varepsilon\downarrow 0$ on both sides of \eqref{epsilon}, and using the facts that all entries of $\adj({\bf R})$ are just sums and differences of minors (determinants of submatrices), and determinants are continuous functions of ${\bf R}$ (in the natural sense).
% we obtain
%\begin{align}
%\det({\bf R}+\vec{y}\vec{w}^{\mathrm{T}})=\det({\bf R})+\vec{w}^{\mathrm{T}}\adj({\bf R})\vec{y},
%\end{align}
%as required.
\hfill\Qed

\noindent {\em Proof of Proposition \ref{prp:all_equal_stable}}.   
When $W(x)=x^{\alpha}$, \eqref{Fdef} becomes
\begin{align*}
F(\vec{v})_i=-v_i+\sum_{A\ni i}p_A\frac{v_i^{\alpha}}{\sum_{j \in A}v_j^{\alpha}},
\end{align*}
so that
\begin{align}
\label{Dalpha_ii}
D_{i,i}(\vec{v})=
-1+\alpha v_i^{\alpha-1}\sum_{A\ni i}p_A
\frac{\sum_{j \in A}v_j^{\alpha}-v_i^{\alpha}}{\left(\sum_{j \in A}v_j^{\alpha}\right)^2},
\end{align}
and, for $k\ne i$,
\begin{align}
\label{Dalpha_ik}
D_{i,k}(\vec{v})=-\alpha v_{k}^{\alpha-1}v_i^{\alpha}\sum_{A \ni i,k} p_A\frac{1}{\left(\sum_{j \in A}v_j^{\alpha}\right)^2}.
\end{align}
When $\vec{v}=\vec{1}/n$, this reduces to
\begin{align}
\label{D_vec1n}
D_{i,i}(\vec{1}/n)=&
%-1+\alpha n^{1-\alpha}\sum_{A\ni i}\frac{p_A}{1-p_{\varnothing}}
%\frac{\sum_{j \in A}n^{-\alpha}-n^{-\alpha}}{\left(\sum_{j \in A}n^{-\alpha}\right)^2}\\
%=&
-1+\alpha n\sum_{A\ni i}p_A
\frac{|A|-1}{|A|^2},
\\
D_{i,k}(\vec{1}/n)=&
%-\alpha n^{1-\alpha}n^{-\alpha}\sum_{A \ni i,k} \frac{p_A}{1-p_{\varnothing}}\frac{1}{\left(\sum_{j \in A}n^{-\alpha}\right)^2}\\
%=&
-\alpha n\sum_{A\ni i,k}p_A
\frac{1}{|A|^2}.
\end{align}

Assume that Condition \ref{cond:strong_symmetry} holds. Then, \eqref{D_vec1n} can be written as
%the above reduces to
%\begin{align}
%\label{Dalpha}
%D_{i,k}(\vec{v})=\begin{cases}
%-1+\alpha v_i^{\alpha-1}\sum_{m=1}^n\frac{p_m}{1-p_{\varnothing}}
%\sum_{A:|A|=m,i\in A}\frac{\sum_{j \in A}v_j^{\alpha}-v_i^{\alpha}}{\left(\sum_{j \in A}v_j^{\alpha}\right)^2}, & \text{ when }k=i\\
%-\alpha v_{k}^{\alpha-1}v_i^{\alpha}\sum_{m=2}^n \frac{p_m}{1-p_{\varnothing}}\frac{1}{\left(\sum_{j \in A}v_j^{\alpha}\right)^2}, &\text{ otherwise}.
%\end{cases} 
%\end{align}
%When $\vec{v}=\vec{1}/n$, \eqref{Dalpha} reduces to
\begin{align}
\label{Dalpha_baricentre}
\nonumber
D_{i,i}(\vec{1}/n)=&
-1+\alpha n\sum_{m=1}^np_m
\sum_{A:|A|=m,i\in A}\frac{m-1}{m^2}\\ \nonumber
=&-1+\alpha n\sum_{m=2}^n\frac{p_m}{m^2}
{n-1 \choose m-1}(m-1)\\
\equiv & -1+\beta,\\ \nonumber
D_{i,k}(\vec{1}/n)=&-\alpha n\sum_{m=2}^n \frac{p_m}{m^2}{n-2 \choose m-2}, \quad \text{ for } k\ne i\\ \nonumber
\equiv&\delta<0.
\end{align}
To compute the eigenvalues of ${\bf D}(\vec{1}/n)$, observe that 
\begin{align*}
{\bf H}\equiv {\bf D}-\lambda{\bf I}=(-(1+\lambda)+\beta-\delta){\bf I}+\vec{1} (\delta\vec{1})^\mathrm{T}.
\end{align*}

Hence, by Lemma \ref{lem:MMDL}, the determinant of ${\bf H}$ is 
\begin{align*}
(-(1+\lambda)+\beta-\delta)^n+\sum_{i=1}^n\delta(-(1+\lambda)+\beta-\delta)^{n-1}.
%\left(1+\frac{n\delta}{\beta-\delta-(1+\lambda)}\right).
\end{align*}
This is equal to zero when 
\begin{align*}
\lambda=\beta-\delta-1, \qquad \text{ or }\qquad \lambda=(n-1)\delta+\beta-1=-1.
\end{align*}
%Since $\delta<0$ the second eigenvalue is less than 0 when the first term is.  
%In the second case we have
%\begin{align}
%\lambda =&-(n-1)\frac{\alpha n}{1-p_{\varnothing}}\sum_{m=2}^n \frac{p_m}{m^2}{n-2 \choose m-2}+\frac{\alpha n}{1-p_{\varnothing}}\sum_{m=2}^n\frac{p_m}{m^2}
%{n-1 \choose m-1}(m-1)-1\\
%= &\frac{\alpha n}{1-p_{\varnothing}}\sum_{m=2}^n\frac{p_m}{m^2}\left[{n-1 \choose m-1}(m-1)-(n-1){n-2 \choose m-2}\right]-1\\
%= &0-1<0.
%\end{align}

%In the first case, we have
The first eigenvalue satisfies
\begin{align*}
\lambda=&\alpha n\sum_{m=2}^n\frac{p_m}{m^2}
{n-1 \choose m-1}(m-1)+\alpha n\sum_{m=2}^n \frac{p_m}{m^2}{n-2 \choose m-2}-1\\
=&\alpha n^2\sum_{m=2}^n\frac{p_m}{m^2}{n-2 \choose m-2}-1,
\end{align*}
which is continuous and increasing in $\alpha$, and is $<0$ if and only if 
\begin{align*}
\alpha<\frac{1}{n^2\sum_{m=2}^n\frac{p_m}{m^2}{n-2 \choose m-2}},
\end{align*}
as required.\hfill\Qed

\subsection{Proof of Theorem \ref{thm:subgraphs}}
Under Condition \ref{cond:Vuniform}, $p_A=1/|V|$ for every $A$ that is the set of edges incident to some vertex, and of course every edge $e$ is an edge in exactly two such $A$.
%This is essentially an application of Lemma \ref{lem:reduction}, but given the importance of the result (and that we did not prove Lemma \ref{lem:reduction}), we give an explicit proof.
%Assume Condition $(\alpha)$.  
For a vertex $x$ and an edge $e$ write $x\in e$ if $e=(x,x')=(x',x)$ for some $x'\in V$ (i.e.~if $x$ is an endvertex of $e$).  Then the equilibrium equation for $e \notin {\bf E}$ is $0=0$, while for $e\in E_j$ it is
\begin{align}
v_e=&\sum_{A \ni e}p_A\frac{v_\e^{\alpha}}{\sum_{e'\in A}v_{e'}^{\alpha}}
=\sum_{\stackrel{x\in V:}{x \in e}}\frac{1}{|V|}\frac{v_\e^{\alpha}}{\sum_{\stackrel{e'\in E:}{x\in e'}}v_{e'}^{\alpha}}\label{subs0}\\
=&\sum_{\stackrel{x\in V_j:}{x \in e}}\frac{1}{|V|}\frac{v_\e^{\alpha}}{   \sum_{\stackrel{e'\in E_j:}{x\in e'}}v_{e'}^{\alpha}+        \sum_{\stackrel{e'\notin E_j:}{x\in e'}}v_{e'}^{\alpha}},\label{subs1}
\end{align}
where 
%\RvdH{What is $p_A$ here?}
Let $e'\notin E_j$ but $x\in e'$ and $x\in V_j$.  Then we have by definition of ${\bf G}$ that $x\notin V_i$ for $i\ne j$, so $e'\notin E_i$ for any $i$.  This means that $v_{e'}=0$.  It follows that the second sum in the denominator of \eqref{subs1} vanishes, so 
\begin{align}
v_e=&\sum_{\stackrel{x\in V_j:}{x \in e}}\frac{1}{|V|}\frac{v_\e^{\alpha}}{   \sum_{\stackrel{e'\in E_j:}{x\in e'}}v_{e'}^{\alpha}}, \quad \text{ and therefore }\quad
\frac{|V|}{|V_j|}v_e=\sum_{\stackrel{x\in V_j:}{x \in e}}\frac{1}{|V_j|}\frac{(\frac{|V|}{|V_j|}v_e)^{\alpha}}{   \sum_{\stackrel{e'\in E_j:}{x\in e'}}(\frac{|V|}{|V_j|}v_{e'})^{\alpha}},\label{subs1b}
\end{align}
which is \eqref{subs0} for the graph $G_j\ni e$ and appropriately rescaled $v$ components.  This proves the first claim.

For the second claim, note that if $e\in E_j$ and $e'\notin E_j$ then for $\vec{v}$ as in the theorem $F(\vec{v})_e$ does not depend on $v_e'$.  Thus for such $e,e'$ we have $D_{e,e'}=0$.  It follows that ${\bf D}(\vec{v})$ is a block diagonal matrix of the form 
\begin{align*}
{\bf D} = \begin{pmatrix} 
{\bf D^{\sss (1)}} & 0 & \cdots & 0 \\ 0 & {\bf D^{\sss (2)}} & \cdots &  0 \\
\vdots & \vdots & \ddots & \vdots \\
0 & 0 & \cdots &{\bf D^{\sss (k+1)}}
\end{pmatrix},
\end{align*}
where ${\bf D^{\sss (i)}}$ is the ${\bf D}$ matrix for $G_i$ for $i\le k$ and where ${\bf D^{\sss (k+1)}}=-{\bf I}_{|E\setminus {\bf E}|}$, and ${\bf I}_m$ denotes the identity matrix of dimension $m$.  Thus the eigenvalues of ${\bf D}$ are simply those of each ${\bf D^{\sss (i)}}$, $i\in[k+1]$ combined.  Since the eigenvalue of ${\bf D^{\sss (k+1)}}$ is $-1<0$, the result follows.\hfill\Qed

\section{Special cases}\label{sec:special cases}
In this section, we examine some of our examples more carefully, beginning with one of the non-graphical WARMs.
\subsection{Fixed $m$, uniform $A_t$ model}
\label{sec:uniform}
Recall that for this model, defined in Example \ref{exa:fixedm}, at least $n-m+1$ colours must be each chosen a positive proportion of the time.  It is not too hard to prove that with positive probability $m-1$ of the colours are never chosen, from which it follows that with positive probability exactly $n-m+1$ colours are each chosen a positive proportion of the time.

From \eqref{equil_alpha} $\vec{v}$ is an equilibrium if and only if
\begin{align}
v_i=&{n \choose m}^{-1}\sum_{\stackrel{A\ni i:}{|A|=m}}\frac{v_i^{\alpha}}{\sum_{j\in A}v_j^{\alpha}}.\label{equil_fixed}
\end{align}
The first claim of Corollary \ref{cor:all_equal_stable} follows directly from Proposition \ref{prp:all_equal_stable} with $p_m={n \choose m}^{-1}$ for each of the ${n \choose m}$ subsets of size $m$.   The following extension of this result could be obtained from Lemmas \ref{lem:reduction} and \ref{lem:symmetry_ones}, and  Proposition \ref{prp:all_equal_stable}, by keeping track of the $p'_s$ for various values of $s$ after $n-k$ colours have been removed.  However, a direct proof is also not too hard, as we will see in the following.

Here and elsewhere, we use the notation $(u)_k$ to denote the vector $(u,\dots, u)\in \R^k$.

\begin{COR}[Stability in the fixed $m$, uniform $A_t$ model]
\label{cor:cond_alpha}
Let $k\ge n-m+1$.  Then $\vec{v}=((1/k)_k,(0)_{n-k})\in \mc{E}_{\alpha}$ for all $\alpha$, and $\vec{v}\in \mc{S}_{\alpha}$ if and only if 
\[\alpha<\dfrac{\displaystyle{n \choose m}}{k^2 \displaystyle{\sum_{r=(m-k)\vee 0}^{\sss (n-k)\wedge (m-2)}{n-k \choose r}\frac{1}{(m-r)^2}{k-2\choose m-r-2}}},\]
while $\vec{v}$ is critical if equality holds.
\end{COR}
\proof
With $\vec{v}$ of the given form we have that \eqref{Dalpha_ii}-\eqref{Dalpha_ik} reduces to $D_{i,i}=-1$ and $D_{i,\ell}=D_{\ell, i}=0$ for $i>k$. 
%$D_{i,\ell}=-1$ if $i>k$ or $\ell>k$, 
For $i\le k$, using the fact that $n-k\le m-1$, and that $m-r-1=0$ if $r=m-1$ we have with $s'=(n-k)\wedge (m-2)$,
\begin{align*}
D_{i,i}=&-1+{n \choose m}^{-1}\alpha  k    \sum_{r=s}^{s'}{k-1\choose m-r-1}{n-k \choose r}\frac{m-r-1}{(m-r)^2},\\
%=&-1+{n \choose m}^{-1}\alpha    \sum_{r=0}^{\sss (n-k)\wedge (m-2)}{k\choose m-r}{n-k \choose r}\frac{m-r-1}{m-r}\\
D_{i,\ell}=&-{n \choose m}^{-1}\alpha  k  \sum_{r=s}^{s'}{k-2\choose m-r-2}{n-k \choose r}\frac{1}{(m-r)^2}.
\end{align*}
%Thus, $\det({\bf D}-\lambda{\bf I})$ is of the form 
%%\begin{align}
%$(-(1+\lambda))^{n-k}\det {\bf B}$,
%%\end{align}
%where ${\bf B}=a{\bf I}+b\vec{1}\vec{1}^\mathrm{T}$ is a $k\times k$ matrix.  
Using Lemma \ref{lem:MMDL},
\begin{align*}
\det({\bf D}-\lambda{\bf I})=(-(1+\lambda))^{n-k})(a^{k}+ba^{k-1}k)=(-(1+\lambda))^{n-k}a^{k-1}(a+bk),
\end{align*}
with 
\begin{align*}
b=&-{n \choose m}^{-1}\alpha  k  \sum_{r=s}^{s'}{k-2\choose m-r-2}{n-k \choose r}\frac{1}{(m-r)^2}\\
a=&-(1+\lambda)+{n \choose m}^{-1}\alpha  k    \sum_{r=s}^{s'}{k-1\choose m-r-1}{n-k \choose r}\frac{m-r-1}{(m-r)^2}-b\\
=&-(1+\lambda)+
%\begin{cases}
{n \choose m}^{-1}\alpha  k^2    \sum_{r=s}^{s'}{n-k \choose r}\frac{1}{(m-r)^2}{k-2\choose m-r-2}.
% & \text{ if }k>n-m+1\\
%{n \choose m}^{-1}\alpha  k    \sum_{r=0}^{s}{k-1\choose m-r-1}{n-k \choose r}\frac{m-r-1}{(m-r)^2}-
%\end{cases}
\end{align*}
The term $a+bk$ can be computed explicitly and equals $-(1+\lambda)$, i.e., $\lambda=-1$.
%\begin{align}
%-(1+\lambda)-\alpha k {n \choose m}^{-1} \sum_{r=0}^{\sss (n-k)\wedge m}{n-k\choose r}\frac{1}{(m-r)^2}{k-2\choose m-r-2},
%\end{align}
%so setting $a+bk=0$ ensures that $\lambda<0$.

Note that $k-1\ge 1$ since $m<n$, so it remains to consider the case $a=0$, i.e.,
\begin{align*}
\lambda=-1+{n \choose m}^{-1}\alpha  k^2    \sum_{r=s}^{s'}{n-k \choose r}\frac{1}{(m-r)^2}{k-2\choose m-r-2},
\end{align*}
which is continuous and increasing in $\alpha$ and is negative if and only if 
\begin{align*}
\alpha<\dfrac{{n \choose m}}{k^2 \sum_{r=s}^{s'}{n-k \choose r}\frac{1}{(m-r)^2}{k-2\choose m-r-2}}.
\end{align*}
\qed

%\begin{LEM}\label{lem:conv_fixed_m}
%Fix $n$, $m$.  If $\alpha>1$ then
%\[\P(\text{exactly }n-m+1 \text{ colours are chosen infinitely often })>0.\]
%\end{LEM}
%%%%%%%%%%%%%%%%%%%%%%%%%%%%%%%%%%%%%

\subsection{Star graph}
\label{sec:star}
Throughout this section, we consider a WARM star graph on $n$ edges.  
First we describe the situation where $n=2$ (which is the same as the simplest line graph, and which also corresponds (after a time-change) to Example \ref{exa:bernoulli} with $n=2$ and $p=1/2$).
\begin{THM}[Equilibria and stability for star graph with two edges]
\label{thm_line_graph_two_edges}
For the star graph with two edges the following is true:  For $\alpha=3$,  $\Ea=\{(1/2,1/2)\}$   and this equilibrium is critical, while 
for every $\alpha\neq 3$ there exists a unique $(v,u)\in \Sa$, where $v=v(\alpha)\ge 1/2$. Moreover, 
$v(\alpha)$ is a continuous function of $\alpha$, that is strictly increasing for $\alpha>3$ from $v(3)=1/2$ to $v(+\infty)=2/3$, and such that $v(\alpha)=1/2$ for $\alpha<3$.
\end{THM}
The main result of this section is the following, which will be proved via a sequence of lemmas:

\begin{THM}[Equilibria and stability for star graphs]
\label{thm:general_star_graph}
There exist $\tilde \alpha(k,n)\in (1,n+1)$ (for $k\in [n]$) such that the only equilibria for the star graph with $n\ge 2$ edges are given by
\begin{itemize}
\item[(i)] $(1/n)_n$ for $\alpha>1$; and (with $v>u$) 
%\item[(i)] $(1,1,\dots,1)/n$ for $\alpha>1$;
\item[(ii)] $((v)_k, (u)_{n-k})$ for $1\le k <n/2$ and $\alpha>\tilde \alpha(k,n)$, with 
%$\tilde \alpha=\alpha(k,n)\in (1,n+1)$, and 
$v(\alpha)$ increasing in $\alpha$;  
\item[(iii)] $((v)_k, (u)_{n-k})$ for $1\le k <n/2$ and $\alpha \in (\tilde\alpha(k,n),n+1)$, with 
$v(\alpha)$ decreasing in $\alpha$;  
\item[(iv)] $((v)_k, (u)_{n-k})$ for $n/2 \le k \le n-1$ and $\alpha>n+1$, with $v(\alpha)$ increasing in $\alpha$.
\end{itemize}
%\RvdH{What is difference between cases (ii) and (ii)? Restrictions on $\alpha$ are the same. I am confused!}
Moreover, $(1/n)_n\in \Sa$ if and only if $\alpha<n+1$ (it is critical when $\alpha=n+1$). Equilibrium  (ii) $\in \Sa$  if and only if $k=1$ and $\alpha>\tilde \alpha(1,n)$ (in which case $v(+\infty)=2/(n+1)$). All other equilibria are not linearly stable.
\end{THM}

%The next theorem describes the general star graph with $n\ge 3$ edges.
%\begin{THM}\label{thm_general_star_graph}
% Assume we have a star graph with $n\ge 3$ edges. Then the following is true:
%${}$
%\begin{itemize}
%\item[(i)] $\vec{1}/n$ is a linearly stable equilibrium if and only if $\alpha<n+1$;
%\item[(ii)] there exists $\alpha^* \in (1,n+1)$ such that for all $\alpha\ge \alpha^*$ there exists a unique linearly stable equilibrium
%$(v,u,\dots,u)$ with $v=v(\alpha)>u$. This equilibrium satisfies $v(\alpha^*)>1/n$, $v(\alpha)$ is a strictly increasing function
%and $v(+\infty)=2/(n+1)$. 
%\item[(iii)] there exist no other linearly stable equilibria.  
%\end{itemize}
%\end{THM}
Recall that for the star graph on $n$ edges, any equilibrium $\vec{v}\in \Ea$ must satisfy
\begin{equation}\label{equi}
v_i = \frac{1}{n+1} + \frac{1}{n+1}\cdot\frac{v_i^{\alpha}}{v_1^{\alpha}+\dots+v_{n}^{\alpha}}, \;\;\;\; 1\le i \le n.
\end{equation}
Then, clearly $\vec{v}\in \Ea$  must satisfy $1/(n+1)< v_i <2/(n+1)$ for each edge $i\in[n]$, therefore all equilibria are internal, and $v_i/v_j\in [1/2,2]$.

\begin{LEM}
\label{lem:star_alpha}
Assume that $\vec{v}\in \Ea$ for the WARM star graph with $n$ edges. Then 
$\vec{v}=(1/n)_n$ or there exist $v>u$ and $k  \in [n-1]$ such that 
$\vec{v}=((v)_k,(u)_{n-k})$.
\end{LEM}
\proof
Assume that $\vec{v}$ is an equilibrium. 
Fix $\delta\in (0,1)$ and consider the set of $\vec{v}$ such that $\sum_{i=1}^n v_i^{\alpha}=\delta$.
%Note that we consider a star graph with $n$ edges and $v_i$ is the probability that edge $i$ is the proportion of times that this edge is reinforced.
%Let $0 < a < 1$ be such that $a=\sum_{i=1}^n v_i^{\alpha}$. 
Define a function $f\colon(0,1)\mapsto \R$ by 
\begin{equation}
f(x) = x^{-1}(1+\delta^{-1}x^{\alpha}).\label{fdef}
\end{equation}
Then \eqref{equi} is equivalent to $f(v_i) = n+1$.  Since \[f'(x)=-x^{-2}+(\alpha-1)\delta^{-1}x^{\alpha-2}=x^{-2}(  (\alpha-1)\delta^{-1}x^{\alpha}-1),\] 
the function $f$ has an extremum where $x^{\alpha} = \delta/(\alpha-1)$.  Hence, when
\[\delta\in \frac{\alpha-1}{(n+1)^{\alpha}}(1,2^{\alpha}),\]
there is exactly one local extremum in $(1,2)/(n+1)$, and otherwise there are no local extrema in $(1,2)/(n+1)$.
%Hence if
%$\delta>(\alpha-1)$, $f$ has no local extrema in $(0,1)$, while if $\delta<(\alpha-1)$, $f$ has exactly one local extremum, at $(\delta/(\alpha-1))^{1/\alpha}\in(0,1)$. 
 It follows that for any $\delta$, there are at most two solutions to $f(x)=n+1$, whence any equilibrium $\vec{v}$ has at most 2 distinct components.
% $\alpha < 1$ or $\alpha > 2$ then the equation 
%$f(v_i)=n$ has one solution in $[0,1]$ for all $a$, hence there exists one %equilibrium in the interior of $\Delta$. For $\alpha \in (1,2)$ there is at %most 2 solutions for $f(v_i)=n+1$ in $[0,1]$ for all $i$. Since this is valid %for all $a$ fixed we obtain the claim.
\hfill\Qed

\begin{LEM}\label{star_graph_lemma_2}
There exist unique equilibria satisfying (ii)-(iv) of Theorem \ref{thm:general_star_graph}.
\end{LEM}
\proof
Assume that $1\le k \le n-1$. Any $\vec{v}\in \Ea$ if and only if it satisfies \eqref{equi}.  For $\vec{v}$ of the form $\vec{v}=((v)_k,(u)_{n-k})$, \eqref{equi} is equivalent to a single equation
\begin{equation}\label{eqn_equi_main}
u=\frac{1}{n+1}+\frac{1}{n+1} \cdot \frac{u^{\alpha}}{kv^{\alpha}+(n-k)u^{\alpha}},
\end{equation}
plus the balance equation $kv+(n-k)u=1$. We introduce a new variable 
\begin{equation*}
t=\ln(v/u)=\ln\left(\frac{1}{k} \left(\frac{1}{u}-n+k\right)\right).
\end{equation*}
From the above formula it follows that
\begin{equation*}
u=\frac{1}{n+k(\e^t-1)},
\end{equation*}
and \eqref{eqn_equi_main} gives us
\begin{equation}\label{eqn_t_n1}
\frac{1}{n+k(\e^t-1)}=\frac{1}{n+1}+\frac{1}{n+1} \cdot \frac{1}{n-k+k\e^{\alpha t}}.
\end{equation}
Solving the above equation for $\e^{\alpha t}$ we obtain
\begin{equation}\label{star_lemma_2_eqn1}
\e^{\alpha t}=\frac{n+1-k}{k} \cdot \frac{\e^t-\frac{n-k}{n-k+1}}{\frac{1+k}{k}-\e^t}.
\end{equation}
Let us define $a:=(n-k)/(n-k+1)$, $b:=(1+k)/k$  and
\begin{equation}\label{def_fkn}
f_{k,n}(t):=\ln \left( \frac{n+1-k}{k} \cdot \frac{\e^t-a}{b-\e^t}\right), \;\;\; \ln(a)<  t<\ln(b), \; 1\le k \le n-1. 
\end{equation}
Then, \eqref{star_lemma_2_eqn1} is equivalent to 
\begin{equation}\label{eqn_main_t_fkn}
\alpha t= f_{k,n}(t). 
\end{equation}
\begin{figure}
\centering
\captionsetup{width=0.8\textwidth}
\FIGS{\includegraphics[height=6cm]{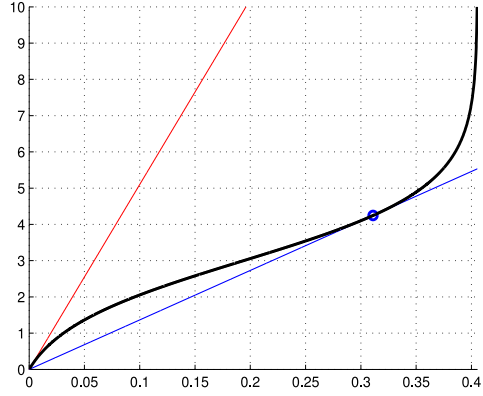}}
\caption{{\small Solving equation $\alpha t= f_{k,n}(t)$ when $k<n/2$. The black curve is $y=f_{k,n}(t)$. The blue line is $y=\tilde \alpha t$
where $\tilde \alpha=\alpha(k,n)$ 
and the red line is $y=(n+1)t$.}}
\label{fig2}
\end{figure}

Let us investigate the function $t\mapsto f_{k,n}(t)$ in more detail. We compute 
\begin{align*}
f_{k,n}'(t)=&\frac{a}{\e^{t}-a}+\frac{b}{b-\e^t}, \\
f_{k,n}''(t)=&\left[-\frac{a}{(\e^t-a)^2}+\frac{b}{(b-\e^t)^2}\right]\e^t=\frac{(b-a)\e^t (\e^{2t}-ab)}{(\e^t-a)^2(b-\e^t)^2}.
\end{align*}
From these equations, we obtain the following facts:
\begin{itemize}
\item[(i)] $f_{k,n}(t)$ is an increasing function and $f_{k,n}(t)\to +\infty$ as $t \uparrow \ln(b)$;
\item[(ii)] $f_{k,n}(0)=0$ and $f_{k,n}'(0)=n+1$;
\item[(iii)] $f_{k,n}(t)$ is concave for $t\in (\ln(a),\tilde t)$ and convex for $t\in (\tilde t,\ln(b))$, where
$\tilde t:=\ln(ab)/2$;
\item[(iv)] $f_{k,n}''(0)=(2k-n)(n+1)$;
\item[(v)]  The inflection point $\tilde t$ satisfies $\tilde t\le 0$  if $k\ge n/2$ and $\tilde t>0$ if $k<n/2$. 
\item[(vi)] For all $t\in (\ln(a),\ln(b))$ we have $f_{k,n}'(t)\ge f_{k,n}'(\tilde t)=\frac{\sqrt{b}+\sqrt{a}}{\sqrt{b}-\sqrt{a}}>1$.
\end{itemize}

Let us first consider the case when $k<n/2$. Then the function $t\mapsto f_{k,n}(t)$ is concave on $(0,\tilde t)$ and 
convex on $(\tilde t, \ln(b))$. The graph of $f_{k,n}(t)$ is shown in Figure \ref{fig2}. Note that there exists a unique $\tilde\alpha(n,k)$ such that the straight line $y=\tilde \alpha t$ is tangent to $y=f_{k,n}(t)$ at some point $\tilde t$ (see the 
blue line in Figure \ref{fig2}). Since $f_{k,n}(0)=n+1$ and $f_{k,n}(t)$ is an increasing function, we see that $\tilde \alpha<n+1$, and item
(vi) above shows that $\tilde \alpha>1$. 
It is clear that: for $\alpha>\tilde\alpha$ there is a solution $t_2(\alpha)$ to \eqref{eqn_main_t_fkn} that is an increasing function of $\alpha$; for $\alpha\in (\tilde\alpha,n+1)$ there is another solution $t_1(\alpha)<\tilde t<t_2(\alpha)$ that is decreasing in $\alpha$; there are no other solutions to \eqref{eqn_main_t_fkn}.
%$\alpha \in (\tilde \alpha, n+1)$ we have two solutions to
%equation \eqref{eqn_main_t_fkn}, $t_1(\alpha)$ and $t_2(\alpha)$, such that $t_1(\alpha)<\tilde t<t_2(\alpha)$, with $t_1(\alpha)$ and $t_2(\alpha)$ being decreasing and increasing functions of $\alpha$ respectively. \MH{It is also clear that for $\alpha\ge n+1$ we have exactly one solution $t_2(\alpha)$
This demonstrates both the existence and uniqueness of equilibria satisfying (ii) and (iii) of Theorem \ref{thm:general_star_graph} respectively.  

When $k\ge n/2$ the situation is simpler, as the function $t\mapsto f_{k,n}(t)$ is convex on $(0,\ln(b))$. Since $f_{k,n}'(0)=n+1$ we see that 
for every $\alpha>n+1$ there exists a unique positive solution to  \eqref{eqn_main_t_fkn}, and that this solution is increasing in $\alpha$. See Figure \ref{fig3}.  
Finally, note that when $k=1$, $t\uparrow \log(b)$ as $\alpha\uparrow \infty$ implies that $v(n-1)/(1-v) \uparrow b=2$ which in turn implies that $v\uparrow 2/(n+1)$.
\hfill\Qed
\begin{figure}
\centering
\captionsetup{width=0.8\textwidth}
\FIGS{\includegraphics[height =6cm]{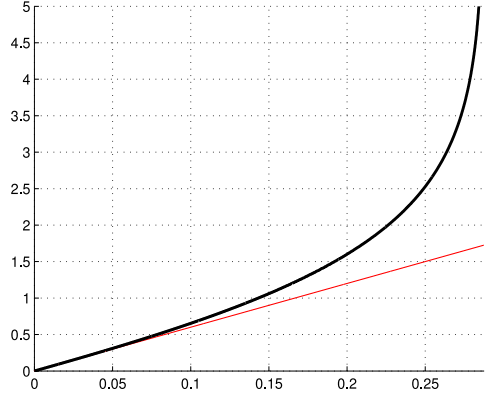}}
\caption{{\small Solving equation $\alpha t= f_{k,n}(t)$ when $k\ge n/2$. The black curve is $y=f_{k,n}(t)$. The red line is $y=(n+1)t$.}}
\label{fig3}
\end{figure}

For $\vec{v}\in \R^n$ and $a>0$, write $\vec{v}^{a}=(v_1^{a},\dots, v_n^a)$, so that e.g.~$((v)_k,(u)_{n-k})^{a}=((v^{a})_k,(u^{a})_{n-k})$.

\begin{LEM}\label{lem:cond equi}
Assume  $\vec{v}=((v)_k,(u)_{n-k})\in \Ea$ for some $1\le k \le n-1$ and $v>u$. 
Let $\eta=kv^{\alpha}+(n-k)u^{\alpha}$ and $\xi=\alpha (n+1)^{-1}\eta^{-2}$. Then $\vec{v}\in \Sa$ (critical if equality holds below) if and only if 
\begin{align*}
k=1 &\quad\qquad \text{and}\quad\qquad \xi(uv)^{\alpha-1}<1, \quad \text{or}\\
k\ge 2 &\quad\qquad \text{and}\quad\qquad  v<\frac{\alpha}{(\alpha-1)(n+1)}.
\end{align*}
\end{LEM}
\proof 
%Let  $\eta=(kv^{\alpha}+(n-k)u^{\alpha})$ and $\xi=\alpha (n+1)^{-1}\eta^{-2}$.  
%Then t
The matrix ${\bf D}$ of partial derivatives has entries
\begin{align*}
D_{ii}=&-1+\xi\times\begin{cases}
v^{\alpha-1}\left((k-1)v^{\alpha}+(n-k)u^{\alpha}\right) & \text{ if }i\le k,\\
u^{\alpha-1}\left(kv^{\alpha}+(n-k-1)u^{\alpha}\right) & \text{ if }i>k,
\end{cases}\\
=&-1+\xi\times\begin{cases}
v^{\alpha-1}\eta -v^{2\alpha-1}& \text{ if }i\le k,\\
u^{\alpha-1}\eta-u^{2\alpha-1}& \text{ if }i>k,
\end{cases}\\
D_{ij}=&-\xi\times \begin{cases}
v^{2\alpha-1}  & \text{ if }i,j\le k,\\
v^{\alpha}u^{\alpha-1}  & \text{ if }i\le k<j,\\
v^{\alpha-1}u^{\alpha}   & \text{ if }j\le k<i,\\
u^{2\alpha-1}  & \text{ if }i,j>k.
\end{cases}
\end{align*}
Let
\begin{align}
\vec{x}=&\vec{v}^{\alpha}, \quad \text{and}\quad \vec{w}=-\xi\vec{v}^{\alpha-1}\label{uwdef}.
%(v^{\alpha},\dots,v^{\alpha},u^{\alpha},\dots, u^{\alpha}), \quad \text{ and }\label{udef}\\
%\vec{w}=&-\xi( v^{\alpha-1},\dots, v^{\alpha-1}, u^{\alpha-1},\dots, u^{\alpha-1}).\label{wdef}
\end{align}
%\RvdH{Is this not a notation clash? $u$ already has a meaning!}

%\begin{align}
%\vec{u}=&(1,\dots,1,(u/v)^{\alpha},\dots, (u/v)^{\alpha}), \quad %\text{ and }\label{udef}\\
%\vec{w}=&-\xi( v^{2\alpha-1},\dots, v^{2\alpha-1}, %v^{\alpha}u^{\alpha-1},\dots, %v^{\alpha}u^{\alpha-1}).\label{wdef}
%\end{align}
Let ${\bf Z}$ be a diagonal matrix with $Z_{ii}=D_{ii}+z_i$, where $\vec{z}=-\lambda\vec{1}+\xi \vec{v}^{2\alpha-1}$.  
%(v^{2\alpha-1},\dots,  v^{2\alpha-1},u^{2\alpha-1},\dots, u^{2\alpha-1})$.  
Then 
%$Z_{ii}=-(1+\lambda)+\xi \vec{v}^{\alpha-1}$
\begin{align*}
Z_{ii}=-(1+\lambda)+\xi \eta \begin{cases}
v^{\alpha-1}, & \text{ if }i\le k,\\
u^{\alpha-1}, & \text{ if }i>k,
\end{cases}
\end{align*}
and ${\bf D}-\lambda {\bf I}={\bf Z}+\vec{x}\vec{w}^\mathrm{T}$.
It follows from Lemma \ref{lem:MMDL} that
\begin{align*}
\det({\bf D}-\lambda {\bf I})=&\det({\bf Z})+\vec{w}^{\mathrm{T}}\adj({\bf Z})\vec{x}\\
=&Z_{1,1}^kZ_{n,n}^{n-k}-\xi v^{2\alpha-1}kZ_{1,1}^{k-1}Z_{n,n}^{n-k}-\xi u^{2\alpha-1}(n-k)Z_{1,1}^{k}Z_{n,n}^{n-k-1}\\
=&Z_{1,1}^{k-1}Z_{n,n}^{n-k-1}\left(Z_{1,1}Z_{n,n}-\xi v^{2\alpha-1}kZ_{n,n}-\xi u^{2\alpha-1}(n-k)Z_{1,1}\right).
\end{align*}
%=&(-(1+\lambda)\xi\eta)^n v^{k(\alpha-1)}u^{\sss (n-k)(\alpha-1)}+k
%\end{align}
After a lot of simplifying, using the definition of $\eta$ and that $kv+(n-k)u=1$ we get that the term in brackets is zero if and only if 
\begin{align*}
(1+\lambda)^2-(1+\lambda)\xi (uv)^{\alpha-1}=0,
\end{align*}
i.e.~if and only if $\lambda=-1$ or $\lambda=-1+\xi(uv)^{\alpha-1}$.  The latter is $<0$ precisely when $\xi(uv)^{\alpha-1}<1$.

If $k>1$ then we also have an eigenvalue when $Z_{1,1}=0$, for which
$\lambda=-1+\xi \eta v^{\alpha-1}$ is negative when $\xi\eta v^{\alpha-1}<1$.  Similarly if $n-k>1$ then we also have an eigenvalue when $Z_{n,n}=0$ for which 
$\lambda=-1+\xi \eta u^{\alpha-1}$ is negative when $\xi\eta u^{\alpha-1}<1$.

Since $u<v$ we have that $\xi\eta u^{\alpha-1}<1$ if $\xi\eta v^{\alpha-1}<1$.  Next,
\begin{align*}
\eta=&u^{\alpha-1}(kv (v/u)^{\alpha-1}+(n-k)u)
>u^{\alpha-1}(kv+(n-k)u)=u^{\alpha-1}.
\end{align*}
This implies that $\xi(uv)^{\alpha-1}<1$ if $\xi\eta v^{\alpha-1}<1$.  Similarly $v^{\alpha-1}>\eta$ so $\xi\eta u^{\alpha-1}<1$ when $\xi (uv)^{\alpha-1}<1$. Therefore, we have proved that the equilibrium with $k=1$ is linearly stable if and only if $\xi (uv)^{\alpha-1}<1$ 
and the equilibrium with $k\ge 2$ is linearly stable if and only if $\xi\eta v^{\alpha-1}<1$. Since $v$ satisfies 
\begin{equation*}
v=\frac{1}{n+1}+\frac{1}{n+1} \times \frac{v^{\alpha}}{kv^{\alpha}+(n-k)u^{\alpha}}
=\frac{1}{n+1}+\frac{\xi \eta v^{\alpha}}{\alpha},
\end{equation*}
the condition $\xi\eta v^{\alpha-1}<1$ is equivalent to $v<\frac{\alpha}{(\alpha-1)(n+1)}$.
\hfill\Qed

\noindent
{\bf Remark:} The proof of Lemma \ref{lem:cond equi} shows that if $k\ge 2$ and $((v)_k,(u)_{n-k})\in \Sa$ then
$\xi (uv)^{\alpha-1}<1$. This observation will be useful for us later, when we investigate the stability of these equilibria. 
\label{page_remark1}
\vspace{0.5cm}

\begin{LEM}\label{star_graph_lemma_4} Assume that $((v)_k,(u)_{n-k})\in \Ea$ with $v>u$ and $\xi$ and $\eta$ are defined as in Lemma \ref{lem:cond equi}. Then the condition 
$\xi(uv)^{\alpha-1}<1$ is equivalent to $\partial v/\partial \alpha>0$. 
\end{LEM}
\proof
We use the same notation as in the proof of Lemma \ref{star_graph_lemma_2}, that is $t=\ln(v/u)$. Taking the derivative $\partial /\partial \alpha$ of both sides of equation \eqref{eqn_t_n1} we obtain, with $t'= \frac{\d t}{\d \alpha}$,
\begin{equation*}
\frac{\e^t t'}{(n+k(\e^t-1))^2}=\frac{1}{n+1} \cdot \frac{\e^{\alpha t} (t+\alpha t')}{(n-k+k\e^{\alpha t})^2}.
\end{equation*}
Since $t>0$ the above equation gives us
\begin{equation*}
\frac{\e^t t'}{(n+k(\e^t-1))^2}>\frac{\alpha}{n+1} \cdot \frac{\e^{\alpha t} t'}{(n-k+k\e^{\alpha t})^2}.
\end{equation*}
Rewriting this inequality in terms of $u$ and $v$ (recall that $u=1/(n+k(\e^t-1))$ and $\e^t=v/u$) we obtain
\begin{equation*}
uv t'>\frac{\alpha}{n+1} \cdot \frac{(uv)^{\alpha}}{(kv^{\alpha}+(n-k)u^{\alpha})^2}t'
\end{equation*}
which is equivalent to
\begin{equation*}
t'(\xi (uv)^{\alpha-1}-1)<0.
\end{equation*}
Therefore, $\xi(uv)^{\alpha-1}<1$ if and only if $t'>0$, which is equivalent to $\partial v/\partial \alpha>0$ since $t=\log(v/u)$ and $u=(1-kv)/(n-k)$.
\hfill\Qed

\noindent
{\em \bf Proof of Theorems \ref{thm_line_graph_two_edges} and
\ref{thm:general_star_graph}.}
The fact that $(1/n)_n\in \Sa$ if and only if $\alpha<n+1$ is part (iii) of Corollary \ref{cor:all_equal_stable}.  By Lemma \ref{lem:star_alpha} all other equilibria are of the form $((v)_k,(u)_{n-k})$ for some $v>u$, $1\le k\le n-1$.

If $n=2$ then $k=1\ge n/2$, and we have by Lemma \ref{star_graph_lemma_2} that there exists an (unique) equilibrium of the form $((v)_k,(u)_{n-k})$ with $v>u$ if and only if $\alpha>n+1$, and that $v(\alpha)$ is increasing to $2/3$.  This proves Theorem \ref{thm_line_graph_two_edges}.

For $n>2$, if $k=1$ and $\alpha \in (\tilde\alpha(1,n),n+1)$ we have by Lemma \ref{star_graph_lemma_2} that there exist two equilibria of the form $(v,(u)_{n-1})$, one of which has $\partial v/\partial \alpha<0$ and the other $\partial v/\partial \alpha>0$. Lemmas \ref{lem:cond equi} and
\ref{star_graph_lemma_4} tell us that linear stability is equivalent to $\partial v/\partial \alpha>0$, so this shows that only one of these two equilibria is linearly stable. When $\alpha>n+1$ we have a unique equilibrium 
of the form $(v,(u)_{n-1})$, and since $\partial v/\partial \alpha>0$ it is linearly stable. 

Next, let us consider the equilibria corresponding to $k>1$. First, assume that $k\ge n/2$ or $k<n/2$ and $\alpha>n+1$. Then we have only one such equilibrium, which exists for $\alpha>n+1$. However, if $\alpha>n+1$ then $\alpha/((\alpha-1)(n+1))<1/n$, and Lemma \ref{lem:cond equi} tells us that such an equilibrium can not be linearly stable (since $v>u$ implies $v>1/n$). 

Finally, let us consider the case $k<n/2$ and $\alpha \in (\alpha(k,n),n+1)$. In this case we have two equilibria, corresponding to two solutions of equation $\alpha t=f_{k,n}(t)$ (see Figure \ref{fig2}). Let us denote these equilibria 
\begin{equation*}
\vec{v}^{\sss (1)}=\vec{v}^{\sss (1)}(\alpha)=((v^{\sss (1)})_k,(u^{\sss (1)})_{n-k})
\end{equation*}
and similarly for $\vec{v}^{\sss (2)}=\vec{v}^{\sss (2)}(\alpha)$. We assume that $v^{\sss (1)}<v^{\sss (2)}$. 
From the proof of Lemma \ref{star_graph_lemma_2}
we know that $v^{\sss (1)}$ is a decreasing function of $\alpha$ while $v^{\sss (2)}$ is an increasing function of $\alpha$.

From the remark on page \pageref{page_remark1} 
%we know that this implies  $\xi(uv)^{\alpha-1}<1$, 
and Lemma 
\ref{star_graph_lemma_4}, $\vec{v}^{\sss (1)}$ can not be linearly stable since $v^{\sss (1)}(\alpha)$ is decreasing in $\alpha$.

Let us consider the equilibrium $\vec{v}^{\sss (2)}$. If this equilibrium is stable, then 
from Lemma \ref{lem:cond equi}, we find that $v^{\sss (2)}<\alpha/((\alpha-1)(n+1))$. Since $v^{\sss (1)}<v^{\sss (2)}$, we see that 
$v^{\sss (1)}$ also satisfies the condition $v^{\sss (1)}<\alpha/((\alpha-1)(n+1))$, therefore $\vec{v}^{\sss (1)}$ must be a stable equilibrium
due to Lemma \ref{lem:cond equi}. Thus we have arrived at a contradiction (since we know that $\vec{v}^{\sss (1)}$ can not be linearly stable), 
and we conclude that $\vec{v}^{\sss (2)}$ is not linearly stable. 
\qed

\subsection{Triangle graph}
\label{sec:triangle}
Consider a WARM triangle graph, under Condition $(\alpha)$. Equations \eqref{equil_alpha} give us the following
\beq\label{system1}
\nonumber
v_1&=&\frac{1}{3} \frac{v_1^{\alpha}}{v_1^{\alpha}+v_2^{\alpha}}+\frac{1}{3} \frac{v_1^{\alpha}}{v_1^{\alpha}+v_3^{\alpha}}, \\
v_2&=&\frac{1}{3} \frac{v_2^{\alpha}}{v_2^{\alpha}+v_3^{\alpha}}+\frac{1}{3} \frac{v_2^{\alpha}}{v_1^{\alpha}+v_2^{\alpha}}, \\\nonumber
v_3&=&\frac{1}{3} \frac{v_3^{\alpha}}{v_1^{\alpha}+v_3^{\alpha}}+\frac{1}{3} \frac{v_3^{\alpha}}{v_2^{\alpha}+v_3^{\alpha}}.
\eeq

From now on we will list $(v_1,v_2,v_3)$ in the decreasing order: $v_1\ge v_2\ge v_3$. 
\begin{THM}[Equilibira and stability for WARM triangle graph]
\label{thm_triangle_main}
 The only equilibria for the WARM triangle graph are given by
\begin{itemize}
\item[(i)] $(1/3,1/3,1/3)$, for all $\alpha>1$;
\item[(ii)] $(1/2,1/2,0)$, for all $\alpha>1$;
\item[(iii)]  $(v,u,0)$ for $\alpha>3$, where $v>u$ and $v(\alpha)$ increases from $v(3+)=1/2$ to $v(+\infty)=2/3$ (here
$(v,u)$ is an equilibrium for the line/star graph with two edges, see Theorem \ref{thm_line_graph_two_edges});
\item[(iv)] $(v,v,u)$, for $\alpha \in (1,4/3]$, where $v>u$ and $v(\alpha)$ decreases from $v(1+)=1/2$ to $v(4/3-)=1/3$.
\item[(v)] $(v,u,u)$, for $\alpha \ge 4/3$, where $v>u$ and $v(\alpha)$ increases from $v(4/3+)=1/3$ to $v(+\infty)=2/3$; 
\end{itemize}
Their stability properties are listed below:
\beqq
&&\textnormal{\it Equilibrium (i) is linearly stable if and only if $\alpha<4/3$,}\\
&&\textnormal{\it Equilibrium (ii) is linearly stable if and only if $\alpha<3$,}\\
&&\textnormal{\it Equilibrium (iii) is linearly stable for all $\alpha>3$,}\\
&&\textnormal{\it Equilibria (iv) and (v) are not linearly stable,}\\
&&\textnormal{\it The equilibria are critical if and only if equality holds in the above.}
\eeqq
%Consequently (by Theorem \ref{thm:convergence})
\end{THM}

%\RvdH{Is it a good idea to add a corollary that states when there is a unique linearly-stable equilibrium? Then we know that the process converges to this a.s.!} 
The proof will be completed by a sequence of lemmas.

\begin{LEM}\label{lemma_triangle_N1}
There exist equilibria described in items (iv) and (v) in Theorem \ref{thm_triangle_main}.
\end{LEM}
\begin{proof}
Let us consider an equilibrium $(v,u,u)$ with $v>u$. Let us denote $v/u=\e^t$, note that $t>0$.
From the condition $v+2u=1$ we find that $u=(2+\e^t)^{-1}$.
Then equation (2) in \eqref{system1}
gives us 
\beq\label{eqn_Lemma1_n1}
\frac{1}{2+\e^t}=\frac{1}{6}+\frac{1}{3} \cdot \frac{1}{1+\e^{\alpha t}},
\eeq
which can be rewritten in the form
\beqq
\e^{\alpha t}=\frac{3\e^t}{4-\e^t},
\eeqq
which is equivalent to
\beq\label{triangle_graph_eqn1}
h(t):= \ln\left( \frac{3}{4-\e^t}\right)=(\alpha-1) t.
\eeq
One can check that the function $h(t)$ is convex on $t\in (0,\log(4))$ and it satisfies $h(0)=0$ and $h'(0)=1/3$, therefore \eqref{triangle_graph_eqn1} has a positive solution $t=t(\alpha)$ if and only if $\alpha>4/3$ (and this solution is necessarily unique).
The graph of the function $t\mapsto h(t)$ is given in Figure \ref{fig1}. It is clear that $\d t/\d \alpha>0$ (see Figure \ref{fig1}), which implies that $v(\alpha)$ is an increasing function. Finally, $t(4/3)=0$ and $t(+\infty)=\ln(4)$, which gives us $v(4/3+)=1/3$ and $v(+\infty)=2/3$. This completes the proof of part (v) in Theorem \ref{thm_triangle_main}.

Let us now consider an equilibrium $(v,v,u)$ with $v>u$. This case is equivalent to the previous one, except that now we have
$u/v=\e^t$ and $t<0$. One can check that $t$ also must satisfy  \eqref{triangle_graph_eqn1}, and that 
\eqref{triangle_graph_eqn1} has a negative solution if and only if $\alpha \in (1,4/3)$. This solution $t=t(\alpha)$ is unique, and 
it satisfies $\frac{\d t}{\d \alpha}>0$,  which translates into the property that $v(\alpha)=1/(2+\e^t)$ is a decreasing function.
Since $t(4/3)=0$ and $t(1)=-\infty$ we see that $v(1+)=1/2$ and $v(4/3-)=1/3$.
\end{proof}

\begin{figure}
\centering
\captionsetup{width=0.8\textwidth}
\FIGS{\includegraphics[height =6cm]{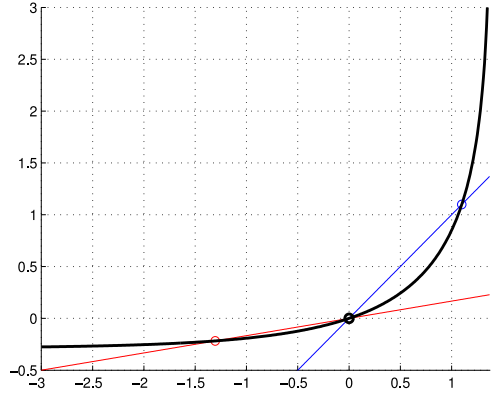}}
\caption{{\small Finding equilibriums of the form $(v,u,u)$ and $(v,v,u)$. The black curve is the graph of the function 
$y=h(t)=\ln(3/(4-\exp(t)))$, the straight lines correspond to graphs of the functions $y=(\alpha-1)t$ for 
$\alpha=2$ (blue) and $\alpha=7/6$ (red).}}
\label{fig1}
\end{figure}

\begin{LEM}\label{lemma_triangle_N2}
For $\alpha \in (1,4/3)$, there are no equilibria other than (i)-(v)  of Theorem \ref{thm_triangle_main}.
\end{LEM}

\begin{proof}
Assume that $(v,u,0)$ is an equilibrium. Then $(v,u)$ is an equilibrium for the line graph with two edges, and Theorem \ref{thm_line_graph_two_edges} shows that for $\alpha \in (1,3]$ the only such equilibrium is $(1/2,1/2)$, and for $\alpha>3$ there are two such equilibria, 
$(1/2,1/2,0)$ and $(v,1-v,0)$. This shows that there do not exist any other equilibria of the form
$(v,1-v,0)$. 
Let us consider $(v_1,v_2,v_3)$, where $v_1\ge v_2\ge v_3>0$. We will show that if $\alpha\in (1,4/3)$ and $(v_1,v_2,v_3)$ is
an equilibrium, then necessarily $v_1=v_2$.  Assume $v_1>v_2$. 
We introduce the new variables $s>0$ and $a\ge 1$
\beqq
\left(\frac{v_2}{v_1}\right)^{\alpha}=\e^{-s}, \qquad\quad \left(\frac{v_2}{v_3}\right)^{\alpha}=a.
\eeqq
Dividing the second equation in \eqref{system1} by the first one we get
\beqq
\frac{v_2}{v_1}=\frac{\frac{1}{1+\left(\frac{v_3}{v_2}\right)^{\alpha}}+\frac{1}{1+\left(\frac{v_1}{v_2}\right)^{\alpha}}}
{\frac{1}{1+\left(\frac{v_2}{v_1}\right)^{\alpha}}+\frac{1}{1+\left(\frac{v_3}{v_1}\right)^{\alpha}}}
\eeqq
In our new notation, this is equivalent to 
\beqq
\e^{-\frac{s}{\alpha}}=\frac{\frac{1}{1+a^{-1}}+\frac{1}{1+\e^s}}{\frac{1}{1+\e^{-s}}+\frac{1}{1+a^{-1}\e^{-s}}}.
\eeqq
We rewrite the above equation as 
\beqq
\e^{\sss (1-\frac{1}{\alpha})s}=\frac{\frac{a}{1+a}+\frac{1}{1+\e^s}}{\frac{1}{1+\e^s}+\frac{a}{1+a\e^{s}}},
\eeqq
and this is equivalent to
\beq\label{eqn_main}
\left(1-\frac{1}{\alpha}\right)s=\ln(1+2a+a\e^s)+\ln(1+a\e^s)-\ln(1+a+2a\e^s)-\ln(1+a)=:f_a(s).
\eeq
We will show that for all $a\ge 1$ and for all $\beta:=(1-1/\alpha) \in [0, 1/4]$, the equation $f_a(s)=\beta s$, $s\ge 0$ has 
a unique solution $s=0$, which implies that $v_1=v_2$. 
We calculate 
\beqq
f_a'(s)=1-\frac{1+2a}{1+2a+a\e^s}-\frac{1}{1+a\e^s}+\frac{1+a}{1+a+2a\e^s},
\eeqq
which shows that
\beqq
4f_a'(s)-1=\frac{6a^3\e^{3s} + 3a^2(a+1)\e^{2s} + (6a^3 - 8a^2 - 4a)\e^{s} - 2a^2 - 3a - 1}
{(1+2a+a\e^s)(1+a\e^s)(1+a+2a\e^s)}.
\eeqq
Note that, for all $s > 0$,
\beqq
6a^3\e^{3s}+(3a^3 - 8a^2 - 4a)\e^{s} > 6a^3\e^{s}+(6a^3 - 8a^2 - 4a)\e^{s}=
4a\e^{s}(3a^2-2a-1) \ge 0, \;\;\; {\textnormal{for all}} \; a \ge 1,
\eeqq
and 
\beqq
3a^2(a+1)\e^{2s}- 2a^2 - 3a - 1 > 3a^3+a^2-3a-1 =(3a+1)(a^2-1) \ge 0, \;\;\; {\textnormal{for all}} \; a \ge 1.
\eeqq
Therefore we have proved that $f_a'(s)>1/4$ for all $a\ge 1$ and all $s > 0$. As a result,
for all $\beta\in(0,1/4)$ it is true that the function $s\mapsto f_a(s)-\beta s$ is strictly increasing, and since $f_a(0)=0$ it shows 
that the only non-negative solution to $f_a(s)=\beta s$ is $s=0$. 
\end{proof}

\begin{LEM}\label{lemma_triangle_N3}
For $\alpha \ge 4/3$ there are no equilibria other than (i)-(v)   of Theorem \ref{thm_triangle_main}.
\end{LEM}

\begin{proof}
 We assume that $\alpha\ge 4/3$ and $v_2>v_3>0$, our goal is to show that this leads to a contradiction. 
 We start by rewriting the second and the third equations in \eqref{system1} as follows
\beqq
\frac{3}{v_2^{\alpha-1}}&=&\frac{a+2b+c}{(a+b)(b+c)}, \\
\frac{3}{v_3^{\alpha-1}}&=&\frac{a+b+2c}{(a+c)(b+c)},
\eeqq
where we have denoted $a=v_1^{\alpha}$, $b=v_2^{\alpha}$ and $c=v_3^{\alpha}$.
Dividing the second equation by the first one we obtain
\beqq
\left(\frac{v_2}{v_3}\right)^{\alpha-1}=\frac{(a+b+2c)(a+b)}{(a+2b+c)(a+c)}.
\eeqq
Some simple algebra shows that the above equation is equivalent to
\beqq
\left(\frac{v_2}{v_3}\right)^{\alpha-1}-1=\frac{b^2-c^2}{(a+2b+c)(a+c)}.
\eeqq
Since $b^2-c^2=(b-c)(b+c)=(b/c-1)(b+c)c$, the previous equation can be rewritten as
\beq\label{eqn1}
\frac{v_2}{v_3}\times \frac{\left(\frac{v_2}{v_3}\right)^{\alpha-1}-1}{\left(\frac{v_2}{v_3}\right)^{\alpha}-1}=\frac{v_2}{v_3}\times \frac{(b+c)c}{(a+2b+c)(a+c)}.
\eeq
Let us denote the expression in the left-hand side \{resp. in the right-hand side\} as $L$ \{resp. $R$\}. Our first goal is to 
prove that $L>1/4$.  Let us denote $w=v_2/v_3$, note that $w>1$. Then 
\beq\label{def_L}
L:=w\frac{w^{\alpha-1}-1}{w^{\alpha}-1}=1-\frac{w-1}{w^{\alpha}-1}. 
\eeq
It is easy to check that for all $\alpha>1$ the function $z\mapsto (z^{\alpha}-1)/(z-1)$ is strictly increasing for $z\in (1,\infty)$, therefore we have
\beqq
\frac{w^{\alpha}-1}{w-1} > \lim\limits_{z\to 1^+} \frac{z^{\alpha}-1}{z-1}=\alpha. 
\eeqq
This implies $(w-1)/(w^{\alpha}-1)<1/\alpha$ and 
\beq\label{inequality_L}
L=1-\frac{w-1}{w^{\alpha}-1}>1-\frac{1}{\alpha}\ge 1/4.
\eeq

Our second goal is to prove that $R\le 1/4$. Let us denote 
$x=v_2/v_1$ and $y=v_3/v_2$, so that $v_2=xv_1$ and $v_3=xyv_1$.  Note that the inequality 
$v_1 \ge v_2 > v_3>0$ implies $0<x\le 1$ and $0<y<1$. We rewrite the right-hand side in \eqref{eqn1} as
\beq\label{def_R}
R:=\frac{v_2}{v_3}\times \frac{(b+c)c}{(a+2b+c)(a+c)}&=&\frac{v_2(v_2^{\alpha}+v_3^{\alpha})v_3^{\alpha-1}}
{(v_1^{\alpha}+2v_2^{\alpha}+v_3^{\alpha})(v_1^{\alpha}+v_3^{\alpha})}\\ \nonumber
&=&
\frac{x^{2\alpha} y^{\alpha-1}(1+y^{\alpha})}{(1+x^{\alpha}(2+y^{\alpha}))(1+x^{\alpha}y^{\alpha})}=:f(x,y).
\eeq
First we check that for all $q>0$ the function  $z\mapsto z^2/((1+z(2+q))(1+zq))$ is increasing for $z>0$, thus
\beqq
\sup\limits_{0<z\le 1} \frac{z^2}{(1+z(2+q))(1+zq)}=\frac{z^2}{(1+z(2+q))(1+zq)} \Big \vert_{z=1}=\frac{1}{(3+q)(1+q)}.
\eeqq
Therefore from the above identity and \eqref{def_R} we obtain
\beq \nonumber\label{eqn2}
R \le 
\sup\limits_{0<t<1} \left[
\sup\limits_{0<s\le 1} f(s,t) \right]&=&
\sup\limits_{0<t<1} t^{\alpha-1}(1+t^{\alpha}) \left[
\sup\limits_{0<s\le 1} \frac{s^{2\alpha}} {(1+s^{\alpha}(2+t^{\alpha}))(1+s^{\alpha}t^{\alpha})}\right]\\
&=&
\sup\limits_{0<t<1} \frac{t^{\alpha-1}}{3+t^{\alpha}}.
\eeq
Consider the function $g(t):=t^{\alpha-1}/(3+t^{\alpha})$. We compute
\beqq
\frac{\d g(t)}{\d t}=\frac{t^{\alpha-2}(3(\alpha-1)-t^{\alpha})}{(3+t^{\alpha})^2}.
\eeqq
Since $3(\alpha-1)\ge 1$ for $\alpha\ge 4/3$, we see that $\d g(t)/ \d t>0$ for $0<t<1$, thus $g(t)$ is increasing for $t\in (0,1)$ and 
\beqq
\sup\limits_{\substack{0<s\le 1\\0<t<1}} f(s,t) =\sup\limits_{0<t<1} \frac{t^{\alpha-1}}{3+t^{\alpha}}=\frac{t^{\alpha-1}}{3+t^{\alpha}} \Big \vert_{t=1}=\frac{1}{4}. 
\eeqq 
The above equation combined with \eqref{eqn1}, \eqref{inequality_L} and \eqref{eqn2} imply 
$1/4<L=R\le 1/4$. 
This shows that our initial assumption $v_2>v_3>0$ can not be true, therefore $v_3=0$ or $v_2=v_3$.
\end{proof}

\begin{LEM}\label{lemma_triangle_N4}
Let us define 
\beqq
\eta:=\frac{\alpha (uv)^{\alpha}}{3(u^{\alpha}+v^{\alpha})^2}.
\eeqq
An equilibrium of the form $(v,u,u)$ or $(u,u,v)$ for $v>u$ is linearly stable if and only if both $\eta<uv$ and $\eta<u-\frac{\alpha}{6}$.
\end{LEM}
\begin{proof}
Assume that $(v_1,v_2,v_3)=(v,u,u)$ and $v\neq u$. The Jacobian matrix is of the form
\begin{align}
{\bf D}=\begin{pmatrix}
-1+\frac{2\eta}{v} & -\frac{\eta}{u} & -\frac{\eta}{u}\\
-\frac{\eta}{v} & -1+\frac{\alpha}{12 u}+\frac{\eta}{u} &  -\frac{\alpha}{12 u}\\
-\frac{\eta}{v} & -\frac{\alpha}{12 u} & -1+\frac{\alpha}{12 u}+\frac{\eta}{u}
\end{pmatrix}.
\end{align}
One can check that 
\beqq
{\textnormal{det}}({\bf D}-\lambda {\bf I})=-(\lambda+1)\left(\lambda+1-\frac{\eta}{uv}(v+2u)\right)
\left(\lambda+1-\frac{\alpha+6\eta}{6u}\right).
\eeqq
Since $v+2u=1$ we see that the eigenvalues are
\beqq
\lambda_1=-1, \;\;\; \lambda_2=-1+\frac{\eta}{uv}, \;\;\; \lambda_3=-1+\frac{\alpha+6\eta}{6u}.
\eeqq
\end{proof}

\begin{LEM}\label{lemma_triangle_N5}
The equilibrium of Theorem \ref{thm_triangle_main}(iv) is not linearly stable.
\end{LEM}
\begin{proof}
Assume that $(v,u,u)$ is an equilibrium, such that $v>u$ and $\alpha>4/3$.  In order to show that 
$(v,u,u)$ is not a linearly stable equilibrium it is enough to prove that that $\eta>u-\alpha/6$ (see Lemma \ref{lemma_triangle_N4}).
Define $r=v/u$. The condition $\eta>u-\alpha/6$ is equivalent to 
\beqq
\frac{1}{2}+\frac{r^{\alpha}}{(1+r^{\alpha})^2}>\frac{3}{\alpha(2+r)}.
\eeqq
This inequality is obvious if $\alpha>2$, so we only need to consider $\alpha \in (4/3,2]$. Let us introduce the new variable $z=r^{\frac{\alpha}{2}}-1$, 
so that $r=(1+z)^{\frac{2}{\alpha}}$. With this notation, we need to prove that for all $\alpha \in (4/3,2]$ and all $z>0$
\beqq
\frac{1}{2}+\frac{(1+z)^2}{(1+(1+z)^2)^2}>\frac{3}{\alpha(2+(1+z)^{\frac{2}{\alpha}})}.
\eeqq
For all $\alpha \in (4/3,2]$ and all $z>0$ we have $(1+z)^{\frac{2}{\alpha}}\ge 1+z$, therefore
\beqq
\frac{3}{\alpha(2+(1+z)^{\frac{2}{\alpha}})}\leq \frac{3}{\alpha(3+z)}< \frac{9}{4(3+z)}.
\eeqq
So it is enough to show that for all $z>0$ 
\beqq
\frac{1}{2}+\frac{(1+z)^2}{(1+(1+z)^2)^2}>\frac{9}{4(3+z)}.
\eeqq
Multiplying both sides by $(1+(1+z)^2)^2(3+z)$ and simplifying the resulting expressions, we obtain that the above inequality is equivalent to
\beqq
2z^5 + 5z^4 + 8z^3 + 12z^2 + 12z>0 \;\;\; {\textnormal{for all}} \; z>0,
\eeqq
which is obviously true. 
\end{proof}

\begin{LEM}\label{lemma_triangle_N6}
The equilibrium of Theorem \ref{thm_triangle_main}(v) is not linearly stable.
\end{LEM}
\begin{proof}
We will show that the first condition of Lemma \ref{lemma_triangle_N4} is not satisfied, that is 
$\eta>uv$ for all $\alpha>4/3$.

Assume that $(v,v,u)$ is an equilibrium. Consider the same parameterization as in the proof of  Lemma \ref{lemma_triangle_N1}: $u/v=\e^t$, 
$v=(2+\e^t)^{-1}$. Note that $t<0$ and from the proof of Lemma \ref{lemma_triangle_N1} we know that $\frac{\d t}{\d \alpha}>0$. 
We consider $t$ as a function of $\alpha$. Equation \eqref{eqn_Lemma1_n1} gives us
\beqq
\frac{\d}{\d \alpha}\left[\frac{1}{2+\e^t}\right]=\frac{\d}{\d \alpha} \left[\frac{1}{6}+\frac{1}{3} \cdot \frac{1}{1+\e^{\alpha t}}\right],
\eeqq
which is equivalent to
\beqq
\frac{\e^t t'}{(2+\e^t)^2}=\frac{1}{3} \cdot \frac{\e^{\alpha t}(t+\alpha t')}{(1+\e^{\alpha t})^2},
\eeqq
where $t':=\frac{\d t}{\d \alpha}$. 
Since $t<0$ and $t'>0$,
\beqq
\frac{\e^t}{(2+\e^t)^2}<\frac{1}{3} \cdot \frac{\e^{\alpha t}\alpha}{(1+\e^{\alpha t})^2}.
\eeqq
Since $\e^t=u/v$ and $(2+\e^t)^{-1}=v$, the above inequality gives us
\beqq
uv<\frac{\alpha (uv)^{\alpha}}{3(u^{\alpha}+v^{\alpha})^2}.
\eeqq
Applying Lemma \ref{lemma_triangle_N4}, we conclude that $(v,v,u)$ is not a linearly stable equilibrium.
\end{proof}

\subsection{Whisker graph}
\label{sec:whisker}
Since we already understand the star-graph setting, let us in this section restrict our attention to whisker graphs that are not star graphs.

For the $(r,s)$-whisker graph (with $r+1+s=n$), $\vec{v}\in \Ea$
%=(v_1,...,v_{n})$ is an equilibrium 
if and only if $\vec{v}$ satisfies (for all $i=1,\dots, n$)
\begin{equation}
0=F(\vec{v})_i = -v_i+\frac{1}{n+1}\begin{cases}
1 + \frac{v_i^{\alpha}}{\delta_r}, & i\le r,\\
v_{r+1}^{\alpha}\left[\frac{1}{\delta_r}+\frac{1}{\delta_s}\right], & i=r+1,\\
1 + \frac{v_i^{\alpha}}{\delta_s}, & r+2\le i\le n,
\end{cases}
\label{whisk1}
\end{equation}
where $\delta_r=\sum_{i=1}^{r+1}v_i^{\alpha}$ and $\delta_s=\sum_{i=r+1}^{n}v_i^{\alpha}$.  Fixing $\delta_r$ and repeating the proof of Lemma \ref{lem:star_alpha} with $f$ given by \eqref{fdef}, we have that for any equilibrium $\vec{v}$ on a whisker graph, $\{v_1,\dots, v_r\}$ has at most 2 distinct elements (only one element when $\delta_r\notin\frac{(\alpha-1)}{(n+1)^{\alpha}}(1,2^{\alpha})$).  Similarly $\{v_{r+2},\dots, v_n\}$ has at most 2 distinct elements (only one element when $\delta_s\notin\frac{(\alpha-1)}{(n+1)^{\alpha}}(1,2^{\alpha})$).   From this we obtain the following lemma:
\begin{LEM}
\label{lem:whisker_equilib}
For all $\alpha>1$, all equilibria for a whisker graph are of the form 
\begin{align}
((v)_{k_r},(u)_{r-k_r},v_{r+1},(v')_{k_s},(u')_{s-k_s}).\label{whisker_equilibria}
\end{align}
\end{LEM}
Note that $v_{r+1}\ge 0$ and all other entries are bounded above and below by $2/(n+1)$ and $1/(n+1)$, respectively.  For such $\vec{v}$, we have that $\delta_r=k_rv^{\alpha}+(r-k_r)u^{\alpha}+v_{r+1}^\alpha$ and similarly  $\delta_s=k_s(v')^{\alpha}+(s-k_s)(u')^{\alpha}+v_{r+1}^\alpha$. 
%\RvdH{I am missing the $v_{r+1}^{\alpha}$ contribution in $\delta_r$ and $\delta_s$??}

Letting $\xi_r=\frac{\alpha}{(n+1)\delta_r^2}$ and $\xi_s=\frac{\alpha}{(n+1)\delta_s^2}$ we have that 
\begin{align}
D_{i,i}=&-1+\frac{\alpha}{n+1} \begin{cases}
v_i^{\alpha-1}\left[\frac{\delta_r-v_i^{\alpha}}{\delta_r^2}\right], & i\le r,\\ v_{r+1}^{\alpha-1}\left[\frac{(\delta_r-v_{r+1}^{\alpha})}{\delta_r^2}+\frac{(\delta_s-v_{r+1}^{\alpha})}{\delta_s^2}\right], & i=r+1, \label{Diiwhisker}\\
v_i^{\alpha-1}\left[ \frac{\delta_s-v_i^{\alpha}}{\delta_s^2}\right], & r+2\le i\le n.
 \end{cases}\\ \nonumber
 =&-1+\begin{cases}
 \xi_rv^{\alpha-1}(\delta_r-v^{\alpha}), & i\le k_r,\\
  \xi_ru^{\alpha-1}(\delta_r-u^{\alpha}), & k_r+1\le i\le r,\\
 \xi_r v_{r+1}^{\alpha-1}(\delta_r-v_{r+1}^{\alpha})+\xi_s v_{r+1}^{\alpha-1}(\delta_s-v_{r+1}^{\alpha}), & i=r+1\\
  \xi_s (v')^{\alpha-1}(\delta_s-(v')^{\alpha}), & r+2\le i\le r+2+k_s,\\
    \xi_s (u')^{\alpha-1}(\delta_s-(u')^{\alpha}), & r+2+k_s\le i\le n.
  \end{cases}
\end{align}
Moreover $D_{i,\ell}=0$ if $i\le r$ and $\ell\ge r+2$ (or vice versa) and otherwise
\begin{align*}
D_{i,\ell}=-v_i^{\alpha}v_{\ell}^{\alpha-1}\begin{cases}
	\xi_r, & i,\ell\le r+1, i\ne \ell,\\	
\xi_s, & i,\ell\ge r+1, i\ne \ell.
	\end{cases}
\end{align*}

Now ${\bf M}\equiv{\bf D}-\lambda {\bf I}$ is of the form 
\begin{align*}
{\bf M}=\begin{pmatrix}
{\bf A} & \vline \,\vec{g}\, \vline & 0\\
\hline
\vec{h}^{\mathrm{T}} & \vline \, a\,  \vline& \vec{t}^{\mathrm{T}}\\
\hline
0 & \vline\, \vec{z}\, \vline & {\bf B}
\end{pmatrix},
\end{align*}
where ${\bf A}\in \R^{r\times r}$ has the same form as the matrix ${\bf D}-\lambda {\bf I}$ in the case of the star-graph on $r$ edges, 
\begin{align*}
\vec{g}^{\mathrm{T}}=-\frac{\alpha v_{r+1}^{\alpha-1}}{(n+1)\delta_r^2}(v^{\alpha},\dots,v^{\alpha},u^{\alpha},\dots,u^{\alpha})\in \R^r
\end{align*} 
 \begin{align*}
 \vec{h}^{\mathrm{T}}=-\frac{\alpha v_{r+1}^{\alpha}}{(n+1)\delta_r^2}(v^{\alpha-1},\dots,v^{\alpha-1},u^{\alpha-1},\dots,u^{\alpha-1})\in \R^r
 \end{align*}
and  $a=D_{r+1,r+1}-\lambda$ etc.  We have that 
 \begin{align*}
\vec{g}=&-\xi_r v_{r+1}^{\alpha-1}\vec{x}_r,\qquad\quad\vec{h}=v_{r+1}^{\alpha} \vec{w}_r,
\end{align*}
where $\vec{x}_r$ and $\vec{w}_r$ are defined as in \eqref{uwdef} (but with $\xi_r$ instead of $\xi$), i.e.,
\begin{align}
\vec{x}_r^{\mathrm{T}}=&(v^{\alpha},\dots,v^{\alpha},u^{\alpha},\dots, u^{\alpha}), \quad \text{ and }\label{whiskerudef}\\
\vec{w}_r^{\mathrm{T}}=&-\xi_r( v^{\alpha-1},\dots, v^{\alpha-1}, u^{\alpha-1},\dots, u^{\alpha-1}).\label{whiskerwdef}
\end{align}
Similarly,
 \begin{align*}
\vec{z}=&-\xi_s v_{r+1}^{\alpha-1}\vec{x}'_s,\qquad\quad\vec{t}=v_{r+1}^{\alpha} \vec{w}'_s.
\end{align*}

\begin{LEM}
\label{lem:determinant_whisker}
The determinant of ${\bf M}$ is given by
\begin{align}
\det({\bf M})=&a\det({\bf A})\det({\bf B})-\left(\det(  {\bf B} )\vec{h}^{\mathrm{T}}\adj( {\bf A}  )\vec{g}+\det(  {\bf A} )\vec{t}^{\mathrm{T}}\adj( {\bf B}  )\vec{z}\right).\label{Hdet1}
\end{align}
\end{LEM}
\proof Firstly note that $\det({\bf M})=\det({\bf H})$, where 
\begin{align*}
{\bf H}=\begin{pmatrix}
{\bf A} & {\bf 0} & \vec{g}\\
{\bf 0} & {\bf B} & \vec{z}\\
\vec{h}^{\mathrm{T}} & \vec{t}^{\mathrm{T}} & a  
\end{pmatrix}.
\end{align*}
Let ${\bf R}= \begin{pmatrix}
{\bf A} & {\bf 0} \\
{\bf 0} & {\bf B}\end{pmatrix}$.
Then using the block matrix form of ${\bf H}$,
\begin{align*}
\det({\bf H})=& (a+1)\det({\bf R})-\det\left({\bf R}+{\vec{g} \choose \vec{z}}(\vec{h}^{\mathrm{T}},\vec{t}^{\mathrm{T}})\right).
\end{align*}

Now by definition of $\adj$ we have that ${\bf R}\adj({\bf R})=\det({\bf R}){\bf I}$, from which it follows easily that for ${\bf R}$ of the form 
$\begin{pmatrix} 
{\bf A} & {\bf 0}\\
{\bf 0} & {\bf B}
\end{pmatrix}$ 
\begin{align*}
\adj( {\bf R})=\begin{pmatrix}
\det(  {\bf B} )\adj( {\bf A}  ) & {\bf 0}\\
{\bf 0} & \det(  {\bf A} )\adj( {\bf B}  )
\end{pmatrix}
\end{align*}
Combining this with Lemma \ref{lem:MMDL}, we arrive at
\begin{align*}
\det({\bf H})=& (a+1)\det({\bf R})-\left(\det({\bf R})+(\vec{h}^{\mathrm{T}},\vec{t}^{\mathrm{T}})\adj({\bf R}){\vec{g} \choose \vec{z}}\right)\\
=&a\det({\bf R})-\left(\det(  {\bf B} )\vec{h}^{\mathrm{T}}\adj( {\bf A}  )\vec{g}+\det(  {\bf A} )\vec{t}^{\mathrm{T}}\adj( {\bf B}  )\vec{z}\right).
\end{align*}
But $\det({\bf R})=\det({\bf A})\det({\bf B})$, yielding \eqref{Hdet1}.
\hfill\Qed

Now we know that ${\bf A}$ and ${\bf B}$ can be written in the form ${\bf A}={\bf Z}+\vec{x}_r\vec{w}^{\mathrm{T}}_r$ and ${\bf B}={\bf Z'}+\vec{x}_s'(\vec{w}'_s)^{\mathrm{T}}$ 
%\RvdH{Correct?}
and where ${\bf Z}$ and ${\bf Z'}$ are diagonal matrices with
\begin{align*}
Z_{ii}=&-(1+\lambda)+\delta_r\xi_r
\begin{cases}
v^{\alpha-1}, & i\le k_r,\\
u^{\alpha-1}, & k_r<i\le r,
\end{cases}\\
Z'_{ii}=&-(1+\lambda)+\delta_s\xi_s
\begin{cases}
(v')^{\alpha-1}, & i\le k_s,\\
(u')^{\alpha-1}, & k_s<i\le n-r-1,
\end{cases}
\end{align*}
for which $\adj({\bf Z})$ is easy to express.  Indeed,
\begin{align*}
\det({\bf Z})=&(-(1+\lambda)+\delta_r\xi_r v^{\alpha-1})^{k_r}(-(1+\lambda)+\delta_r\xi_r u^{\alpha-1})^{r-k_r},\\
\vec{w}_r^{\mathrm{T}}\adj({\bf Z})\vec{u}_r=& -\xi_r
\left(\sum_{i=1}^{k_r}v^{2\alpha-1}\left[(-(1+\lambda)+\delta_r\xi_r v^{\alpha-1})^{k_r-1}(-(1+\lambda)+\delta_r\xi_r u^{\alpha-1})^{r-k_r}\right]\right.\\
& \quad +\left.\sum_{i=k_r+1}^{r}u^{2\alpha-1}\left[(-(1+\lambda)+\delta_r\xi_r u^{\alpha-1})^{k_r}(-(1+\lambda)+\delta_r\xi_r v^{\alpha-1})^{r-k_r-1}\right]\right),
\end{align*}
and if both $k_r\ge 1$ and $r-k_r\ge 1$ this becomes
\begin{align*}
\vec{w}_r^{\mathrm{T}}\adj({\bf Z})\vec{u}_r=& -\xi_r (-(1+\lambda)+\delta_r\xi_r v^{\alpha-1})^{k_r-1}(-(1+\lambda)+\delta_r\xi_r u^{\alpha-1})^{r-k_r-1}\\
&\quad \Big(k_rv^{2\alpha-1}(-(1+\lambda)+\delta_r\xi_r u^{\alpha-1})+(r-k_r)u^{2\alpha-1}(-(1+\lambda)+\delta_r\xi_r u^{\alpha-1})\Big).
\end{align*}
Similarly
\begin{align*}
\det({\bf Z'})=&(-(1+\lambda)+\delta_s\xi_s v^{\alpha-1})^{k_s}(-(1+\lambda)+\delta_s\xi_s (u')^{\alpha-1})^{s-k_s},\\
\vec{w}_s^{\mathrm{T}}\adj({\bf Z'})\vec{u}_s=& -\xi_s\left(\sum_{i=1}^{k_s}(v')^{2\alpha-1}\left[(-(1+\lambda)+\delta_s\xi_s (v')^{\alpha-1})^{k_s-1}(-(1+\lambda)+\delta_s\xi_s (u')^{\alpha-1})^{s-k_s}\right]\right.\\
& \quad +\left.\sum_{i=k_s+1}^{s}(u')^{\alpha-1}\left[(-(1+\lambda)+\delta_s\xi_s (u')^{\alpha-1})^{k_s}(-(1+\lambda)+\delta_s\xi_s (v')^{\alpha-1})^{s-k_s-1}\right]\right).
\end{align*}

The question is whether we can handle the term of the form $\vec{h}^{\mathrm{T}}\adj( {\bf A}  )\vec{g}$.  However, again by Lemma \ref{lem:MMDL},
\begin{align*}
 \vec{h}^{\mathrm{T}}\adj( {\bf A}  )\vec{g}=\det({\bf A}+\vec{g}\vec{h}^{\mathrm{T}})-\det({\bf A}),
 \end{align*}
and we know what to do with $\det({\bf A})$ as above.  On the other hand, since ${\bf A}={\bf Z}+\vec{u}_r\vec{w}_r$ and $\vec{g}=-\xi_r v_{r+1}^{\alpha-1}\vec{u}_r$ and $\vec{h}=v_{r+1}^{\alpha}\vec{w}_r$,
\begin{align*}
{\bf A}+\vec{g}\vec{h}^{\mathrm{T}}={\bf Z}+\vec{u}_r\vec{w}_r^{\mathrm{T}}-\xi_r v_{r+1}^{2\alpha-1}\vec{u}_r\vec{w}_r^{\mathrm{T}}={\bf Z}+(1-\xi_r v_{r+1}^{2\alpha-1})\vec{u}_r\vec{w}_r^{\mathrm{T}}.
\end{align*}
Thus we can express the determinant of ${\bf A}+\vec{g}\vec{h}^{\mathrm{T}}$ as
\begin{align*}
\det({\bf A}+\vec{g}\vec{h}^{\mathrm{T}})=&\det({\bf Z})+(1-\xi_r v_{r+1}^{2\alpha-1})\vec{w}_r^{\mathrm{T}}\adj({\bf Z})\vec{u}_r\\
=& \det({\bf A})-\xi_r v_{r+1}^{2\alpha-1}\vec{w}_r^{\mathrm{T}}\adj({\bf Z})\vec{u}_r,
\end{align*}
since $\det({\bf A})=\det({\bf Z})+\vec{w}_r^{\mathrm{T}}\adj({\bf Z})\vec{u}_r$.
Since we can do the same with the ${\bf B}$ terms we can write an expression for the determinant in terms of all these quantities.

Recall from \eqref{Diiwhisker} that 
\begin{align*}
a=&-(1+\lambda)+\frac{\alpha v_{r+1}^{\alpha-1}}{n+1}
\left[\frac{(\delta_r-v_{r+1}^{\alpha})}{\delta_r^2}+\frac{(\delta_s-v_{r+1}^{\alpha})}{\delta_s^2}\right]\\
=&-(1+\lambda)+v_{r+1}^{\alpha-1}\big(\xi_r(\delta_r-v_{r+1}^{\alpha})+\xi_s(\delta_s-v_{r+1}^{\alpha})\big),
\end{align*}
where $\delta_r-v_{r+1}^{\alpha}=\sum_{i=1}^rv_i^{\alpha}$ and $\delta_s-v_{r+1}^{\alpha}=\sum_{i=r+2}^n v_i^{\alpha}$.  From \eqref{Hdet1} and the above we have established the following lemma:
\begin{LEM}
\label{lem:determinant_whisker_2}
The determinant of ${\bf M}$ satisfies
\begin{align}
\det({\bf M})=&a\left[\det({\bf Z})+\vec{w}_r^{\mathrm{T}}\adj({\bf Z})\vec{u}_r\right]\left[\det({\bf Z'})+\vec{w}_s^{\mathrm{T}}\adj({\bf Z'})\vec{u}_s\right]\label{twostars}\\
&+\left[\det({\bf Z'})+\vec{w}_s^{\mathrm{T}}\adj({\bf Z'})\vec{u}_s\right]\left[\xi_r v_{r+1}^{2\alpha-1}\vec{w}_r^{\mathrm{T}}\adj({\bf Z})\vec{u}_r\right]\label{extra1} \\
&+\left[\det({\bf Z})+\vec{w}_r^{\mathrm{T}}\adj({\bf Z})\vec{u}_r\right]\left[\xi_s v_{r+1}^{2\alpha-1}\vec{w}_s^{\mathrm{T}}\adj({\bf Z'})\vec{u}_s\right].\label{extra2}
\end{align}
\end{LEM}
\subsubsection{Special cases}
If $v_{r+1}=0$ then $a=-(1+\lambda)$, the two terms \eqref{extra1} and \eqref{extra2} vanish and we recover the fact (see Theorem \ref{thm:subgraphs}) that the case $v_{r+1}=0$ is linearly stable if and only if each of the remaining star graphs is linearly stable.

\bigskip

Let us now examine the completely symmetric case $r=s=k_r=k_s$, $v=v'$.
\begin{LEM}
\label{lem:stable_symmetric_whisker}
For the symmetric whisker graph with $r=s=k_r=k_s$,  $\vec{v}=((v)_r,v_{r+1},(v)_r)$ is a linearly stable equilibrium if and only if 
\begin{align*} 
\xi_r v_{r+1}^{\alpha-1}v^{\alpha-1}<1, \quad \text{and in the case $r>1$ also}\quad  \delta_r\xi_r v^{\alpha-1}<1.
\end{align*}
\end{LEM}
\proof  We have that ${\bf Z}={\bf Z'}$ etc., and thus
\begin{align*}
\det({\bf M})=&a\left[\det({\bf Z})+\vec{w}_r^{\mathrm{T}}\adj({\bf Z})\vec{u}_r\right]^2\\
&+2\left[\xi_r v_{r+1}^{2\alpha-1}\vec{w}_r^{\mathrm{T}}\adj({\bf Z})\vec{u}_r\right]\left[\det({\bf Z})+\vec{w}_r^{\mathrm{T}}\adj({\bf Z})\vec{u}_r\right]\\
=&\left[\det({\bf Z})+\vec{w}_r^{\mathrm{T}}\adj({\bf Z})\vec{u}_r\right]\left(a\left[\det({\bf Z})+\vec{w}_r^{\mathrm{T}}\adj({\bf Z})\vec{u}_r\right]+2\left[\xi_r v_{r+1}^{2\alpha-1}\vec{w}_r^{\mathrm{T}}\adj({\bf Z})\vec{u}_r\right]\right).
\end{align*}
Here $\det({\bf Z})=(-(1+\lambda)+\delta_r\xi_r v^{\alpha-1})^{r}$ and 
\begin{align*}
\det({\bf Z})+\vec{w}_r^{\mathrm{T}}\adj({\bf Z})\vec{u}_r
=&(-(1+\lambda)+\delta_r\xi_r v^{\alpha-1})^{r}-r\xi_rv^{2\alpha-1}(-(1+\lambda)+\delta_r\xi_r v^{\alpha-1})^{r-1}\\
=&(-(1+\lambda)+\delta_r\xi_r v^{\alpha-1})^{r-1}\left[(-(1+\lambda)+\delta_r\xi_r v^{\alpha-1})-r\xi_rv^{2\alpha-1}\right]\\
=&(-(1+\lambda)+\delta_r\xi_r v^{\alpha-1})^{r-1}\left[-(1+\lambda)+\xi_r v^{\alpha-1}v_{r+1}^{\alpha}\right],
\end{align*}
so 
\begin{align*}
\lambda=\delta_r\xi_r v^{\alpha-1}-1, \quad \text{and} \quad \lambda=\xi_r v^{\alpha-1}v_{r+1}^{\alpha}-1
\end{align*}
are eigenvalues, with the first of multiplicity $r-1$ (vanishing when $r=1$).

Next
\begin{align*}
a=&-(1+\lambda)+v_{r+1}^{\alpha-1}\big(\xi_r(\delta_r-v_{r+1}^{\alpha})+\xi_s(\delta_s-v_{r+1}^{\alpha})\big)\\
=&   -(1+\lambda)+2rv_{r+1}^{\alpha-1}\xi_rv^{\alpha}, 
\end{align*}
so
\begin{align*}
%&\left(a\left[\det({\bf Z})+\vec{w}_r^{\mathrm{T}}\adj({\bf Z})\vec{u}_r\right]+2\left[\xi_r v_{r+1}^{2\alpha-1}\vec{w}_r^{\mathrm{T}}\adj({\bf Z})\vec{u}_r\right]\right)\\
\det({\bf M})=&		\big(-(1+\lambda)+2rv_{r+1}^{\alpha-1}\xi_rv^{\alpha}\big)				(-(1+\lambda)+\delta_r\xi_r v^{\alpha-1})^{r-1}\left[-(1+\lambda)+\xi_r v^{\alpha-1}v_{r+1}^{\alpha}\right]\\
&-2r\xi_rv^{2\alpha-1}(-(1+\lambda)+\delta_r\xi_r v^{\alpha-1})^{r-1}\xi_r v_{r+1}^{2\alpha-1}\\
%\begin{align}
%&a\det({\bf A})+2\left[\xi_r v_{r+1}^{2\alpha-1}\vec{w}_r^{\mathrm{T}}\adj({\bf Z})\vec{u}_r\right]\\
%=&\left(-(1+\lambda)+2rv_{r+1}^{\alpha-1}\xi_rv^{\alpha}\right)(-(1+\lambda)+\delta_r\xi_r %v^{\alpha-1})^{r-1}\left[-(1+\lambda)+\xi_r v^{\alpha-1}v_{r+1}^{\alpha}\right]\\
%&-2\xi_rv_{r+1}^{2\alpha-1}r\xi_rv^{2\alpha-1}(-(1+\lambda)+\delta_r\xi_r v^{\alpha-1})^{r-1}\\
%=&(-(1+\lambda)+\delta_r\xi_r %v^{\alpha-1})^{r-1}\left[\left(-(1+\lambda)+2rv_{r+1}^{\alpha-1}\xi_rv^{\alpha}\right)%\left[-(1+\lambda)+\xi_r %v^{\alpha-1}v_{r+1}^{\alpha}\right]-2\xi_rv_{r+1}^{2\alpha-1}r\xi_rv^{2\alpha-1}\right]\\
%=&(-(1+\lambda)+\delta_r\xi_r %v^{\alpha-1})^{r-1}\big(1+\lambda\big)\left[(1+\lambda)-2rv_{r+1}^{\alpha-1}\xi_rv^{\alpha}%-\xi_rv^{\alpha-1}v_{r+1}^{\alpha}\right]\\
=&(-(1+\lambda)+\delta_r\xi_r v^{\alpha-1})^{r-1}\big(1+\lambda\big)\left[(1+\lambda)-\xi_r v_{r+1}^{\alpha-1}v^{\alpha-1}\right],
\end{align*}
where we have used $2rv+v_{r+1}=1$.
The corresponding eigenvalues are
\begin{align*}
\lambda=\delta_r\xi_r v^{\alpha-1}-1, \quad \lambda=-1, \quad \text{and}\quad \lambda=\xi_r v_{r+1}^{\alpha-1}v^{\alpha-1}-1,
\end{align*}
with the former not being present when $r=1$.\Qed

We are now ready to state our main result of this section:
\begin{THM}
\label{thm:symmetric_whisker2}
On the symmetric whisker graph, with $r \ge 1$ there exists $\alpha(r)>1$ such that for any $\alpha>\alpha(r)$ there exist two equilibria of the form $((v)_r,u,(v)_r)$,
%\eqref{form_of_the_equilibrium}
 both with $v<u$, exactly one of which is linearly stable.
For the linearly stable equilibrium the function $u(\alpha)$ increases to $u(+\infty)=(r+1)^{-1}$. 
For $\alpha<\alpha(r)$ there do not exist equilibria of the form 
$((v)_r,u,(v)_r)$ with $u>0$ 
\end{THM}
\proof 
To establish the existence of such equilibria we need to show that the equation
\beq\label{equation_whisker_u1}
v=\frac{1}{2(r+1)}+\frac{1}{2(r+1)}\frac{v^{\alpha}}{u^{\alpha}+rv^{\alpha}},
\eeq
or, equivalently, 
\beq\label{equation_whisker_v1}
u=\frac{1}{r+1}\frac{u^{\alpha}}{u^{\alpha}+rv^{\alpha}}, 
\eeq
has a solution $u>0$, $v>0$, satisfying $u+2rv=1$. We define $u/v=\e^t$, then $v=1/(2r+\e^t)$ and we find from \eqref{equation_whisker_u1} that
\beq\label{whisker_eqn1}
\frac{1}{2r+\e^t}=\frac{1}{2(r+1)}+\frac{1}{2(r+1)} \cdot \frac{1}{1+\e^{\alpha t}}.
\eeq
Solving this equation for $\e^{\alpha t}$ we obtain
\beq\label{whisker_eqn0}
\e^{\alpha t}=\frac{(r+1) \e^t}{2-\e^t}
\eeq
which is equivalent to
\beq\label{whisker_eqn1b}
(\alpha-1)t=\ln\left(\frac{r+1}{2-\e^t}\right).
\eeq
The function $h_r(t):=\ln((r+1)/(2-\e^t))$ is convex, increasing and strictly positive on $t\in (-\infty,\ln(2))$. The graph of this function is shown in Figure \ref{fig_hr}. Since the function is convex, increasing and $h_r(0)>0$ it is clear that there exists
$\alpha(r)$ such that 
the equation $h_r(t)=(\alpha-1)t$ will have two solutions for $\alpha>\alpha(r)$ and no solutions for $\alpha<\alpha(r)$.  
\begin{figure}
\centering
\captionsetup{width=0.8\textwidth}
\FIGS{\includegraphics[height =6cm]{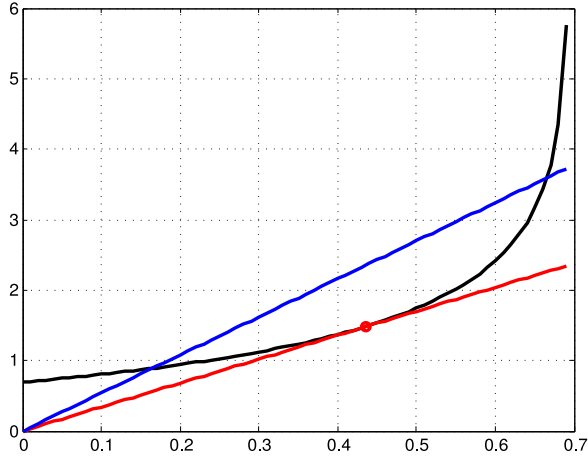}}
\caption{{\small The graph of functions $y=h_r(t)$ (black), $y=(\alpha(r)-1)t$ (red) and
$y=(\alpha-1)t$ for $\alpha>\alpha(r)$ (blue). The point $(t(r), h_r(t(r)))$ is marked by a red circle.}}
\label{fig_hr}
\end{figure}

Let us define for $\alpha>1$
\beq\label{def_ga}
g(\alpha):=(\alpha-1)\ln(2(\alpha-1))-\alpha \ln(\alpha).
\eeq
One can check that $g''(\alpha)>0$, and that $g'(1+)=-\infty$ and $g(+\infty)=+\infty$. The graph of the function 
$\alpha\mapsto g(\alpha)$ is shown in Figure \ref{fig_ga}. For $r\ge 1$ we define $\alpha_*(r)$ to be the unique positive solution to the equation
\beq\label{equation_for_alpha_r}
 g(\alpha)=\ln\left(\frac{1+r}{2}\right), \;
 \textnormal{or equivalently } \; \left[ 2\left(1-\frac{1}{\alpha}\right)\right]^{\alpha-1}=\frac{\alpha}{2}(r+1).
\eeq
  We can see $\alpha(1)$ (the solution to $g(\alpha)=0$) marked by a red circle on Figure
 \ref{fig_ga}. 

\begin{figure}
\centering
\captionsetup{width=0.8\textwidth}
\FIGS{\includegraphics[height =6cm]{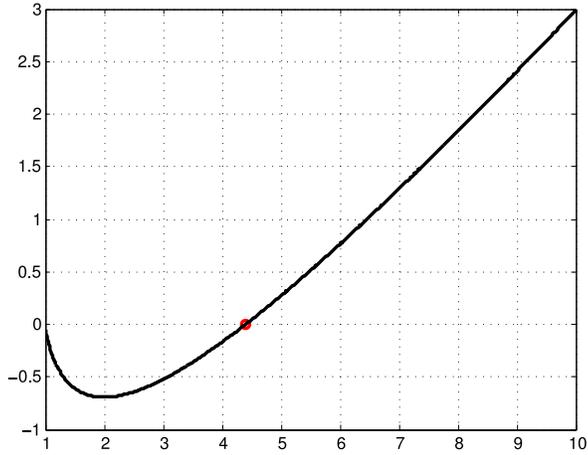}}
\caption{{\small The graph of the function $\alpha\mapsto g(\alpha)$ defined in \eqref{def_ga}.}}
\label{fig_ga}
\end{figure}

%\beqq
%v=\frac{1}{2(r+1)}+\frac{1}{2(r+1)}\frac{v^{\alpha}}{u^{\alpha}+rv^{\alpha}}
%\eeqq
Let us show that $\alpha(r)$ satisfies equation \eqref{equation_for_alpha_r} (i.e.~that $\alpha(r)=\alpha_*(r)$. From the graph in Figure \ref{fig_hr} it is clear that $\alpha(r)$ is characterized by the following system of two equations 
 \beqq
 h_r(t)&=&(\alpha-1)t, \qquad\qquad h_r'(t)=\alpha-1. 
 \eeqq
This system expresses the fact that the graph of the straight line $y=(\alpha-1)t$ must be a tangent line to the curve $y=h_r(t)$ at the point of their intersection $t=t(r)$.  From the equation $h_r'(t)=\alpha-1$ we find that 
\beq\label{eqn_tr}
 \e^t=2\left(1-\frac{1}{\alpha}\right)
\eeq 
and substituting this result into the first equation $h_r(t)=(\alpha-1)t$ (or into the equivalent equation 
\eqref{whisker_eqn0}) we obtain \eqref{equation_for_alpha_r}. 

Thus we have now proved that: (i) For $\alpha<\alpha(r)$ there do not exist equilibria of the form $((v)_r,u,(v)_r)$; (ii) For  all $\alpha>\alpha(r)$ the equation \eqref{whisker_eqn1b} has two solutions, $0<t_1(\alpha)<t_2(\alpha)<\ln(2)$, such that $t_1(\alpha)$ is decreasing in $\alpha$ and $t_2(\alpha)$ is increasing in $\alpha$. These 
two solutions give us two equilibria of the form $((v)_r,u,(v)_r)$ (recall that $u=1/(2r+\e^t)$). 

Next, let us investigate stability of these equilibria. According to Lemma \ref{lem:stable_symmetric_whisker}, the equilibrium is linearly stable if and only if 
\beq\label{condition1}
\frac{\alpha (uv)^{\alpha-1}}{2(r+1)(u^{\alpha}+rv^{\alpha})^2}<1, 
\eeq
and in the case $r=1$,
\beq\label{condition2}
\frac{\alpha v^{\alpha-1}}{2(r+1)(u^{\alpha}+rv^{\alpha})}<1. 
\eeq

Applying the same ideas as in the proof of Lemmas \ref{star_graph_lemma_4} and \ref{lemma_triangle_N6} (taking the derivative $\partial /\partial \alpha$ of equation \eqref{whisker_eqn1b}) we check that the inequality
\eqref{condition1} is equivalent to $\frac{\partial t}{\partial \alpha}>0$. One of the two equlibriums that we have found (the one corresponding to the solution $t_1(\alpha)$) is decreasing in $\alpha$, therefore it can not possibly be a stable equilibrium. 
At the same time the second solution $t_2(\alpha)$ is increasing in $\alpha$, therefore the condition \eqref{condition1} is satisfied. Let us look at the remaining condition 
\eqref{condition2}. Using \eqref{equation_whisker_v1}, we see that this is equivalent to 
\beq\label{inequality_lambda2}
\left(\frac{u}{v} \right)^{\alpha-1}=\e^{\sss (\alpha-1)t_2}>\frac{\alpha}{2}. 
\eeq
Recall that we have denoted the unique solution to the equation $h_r(t)=(\alpha(r)-1)t$ by $t(r)$ (see Figure \ref{fig_hr}).
Note that $\alpha>\alpha(r)$ and $t_2>t(r)$. Inequality \eqref{inequality_lambda2} is satisfied when 
$\alpha=\alpha(r)$ and $t_2=t(r)$, since from formulas  \eqref{equation_for_alpha_r} and \eqref{eqn_tr}  it follows that
\beqq
\e^{\sss (\alpha(r)-1)t(r)}=\left[ 2\left(1-\frac{1}{\alpha(r)}\right)\right]^{\alpha(r)-1}=\frac{\alpha(r)}{2} (r+1)>\frac{\alpha(r)}{2}. 
\eeqq
If we increase $\alpha$ (while keeping $t_2=t(r)$ constant),  then the inequality is still true, as the function on the left-hand side increases faster than the function on the right-hand side. Increasing $t_2$ will only increase the left-hand side, while keeping the right-hand side constant, and the required inequality is still true.  Thus we have proved that the second equilibrium (the one with $v/u=\e^{t_2(\alpha)}$) is linearly stable. 
\qed

%
%\blank{
%
%whisker=function(r=2,s=2,alpha=2,time=10^6){
%	n=r+s+1
%	vertex=sample(1:(n+1),time,replace=T)
%  props=rep(1,n)
%props=c(rep(1,r),3,rep(1,s))
%	for(i in 1:time){
%		if(vertex[i]<r+1){props[vertex[i]]=props[vertex[i]]+1}
%		if(vertex[i]>r+2){props[vertex[i]-1]=props[vertex[i]-1]+1}
%		if(vertex[i]==r+1){chosen=sample(1:(r+1),1,prob=(props[1:(r+1)])^alpha)
%												props[chosen]=props[chosen]+1}
%		if(vertex[i]==r+2){chosen=sample((r+1):n,1,prob=(props[(r+1):n]-1)^alpha)
%												props[chosen]=props[chosen]+1}	
%									}
%	list(p_r=props[1:r]/(time+n),p=props[r+1]/(time+n),p_s=props[(r+2):n]/(time+n))
%	}
%	
%	}			
%																		

\subsection{Complete graph}
\label{sec:complete}

\begin{THM}\label{thm_complete_graph_1_over_n}
Consider a complete graph on $n_v$ vertices and $n:=n_v(n_v-1)/2$ edges. For $n_v=3$, the equilibrium $\vec{1}/n$ is linearly stable if $\alpha<4/3$ (critical if equality holds), and it is linearly unstable if $\alpha>4/3$. For $n_v\ge 4$, the equilibrium $\vec{1}/n$ is linearly unstable. 
\end{THM}
\begin{proof}
The case of $n_v=3$ (triangle graph) was considered in full detail in Theorem \ref{thm_triangle_main}. Let us assume that $n_v \ge 4$. 
Let $K_{n_v}$ be the complete graph on $n_v$ vertices. We recall that the line-graph $L=L(K_{n_v})$ is defined by considering edges of $K_{n_v}$ as vertices of $L$, and the vertices of $L$ are adjacent if and only if the corresponding edges of $K_{n_v}$ are both incident to some vertex in $K_{n_v}$. 
The equations \eqref{D_vec1n} give us 
\begin{align}
D_{i,j}(\vec{1}/n)=
\begin{cases}
-1+\alpha-\frac{\alpha}{n_v-1}, \;\;\; &{\textnormal{ if }} \; i=j, \\
-\frac{\alpha}{2(n_v-1)}, \;\;\; &{\textnormal{ if }} \; i\ne j,  \text{ and }i,j \text{ are both  incident to some vertex }x,\\
0, \;\;\; & \text{ otherwise.}
\end{cases}
\end{align}
Note that 
\begin{align}
{\bf D}=\left(-1+\alpha-\alpha\frac{1}{n_v-1}\right){\bf I}-\frac{\alpha}{2(n_v-1)} {\bf A},
\end{align}
where ${\bf A}$ is the adjacency matrix of $L$. According to \cite[Corollary 1.4.2]{Brouwer}, the matrix ${\bf A}$ has an eigenvalue $-2$ of degree 
$n-n_v$. This shows that the matrix ${\bf D}$ has an eigenvalue 
\beqq
-1+\alpha-\alpha\frac{1}{n_v-1}-\frac{\alpha}{2(n_v-1)} \times (-2)=-1+\alpha>0
\eeqq
 of multiplicity $n-n_v$, and therefore 
$\vec{1}/n$ is a linearly unstable equilibrium. 
\end{proof}

\subsection{Circle graph}
\label{sec:circle}

\begin{LEM}
The equilibrium $\vec{v}=\vec{1}/n$ is linearly stable if and only if $n$ is odd and $\alpha<\cos\left(\tfrac{\pi}{2n} \right)^{-2}$.
\end{LEM}
\proof
For the circle graph with $n$ vertices and edges, we label the edges $\{0,\dots, n-1\}$ around the circle (in the obvious way) and use addition and subtraction $\mod(n-1)$.  Then, $\vec{v}$ is an equilibrium if and only if  
\begin{align}
v_i=\frac{1}{n}\frac{v_i^{\alpha}}{v_i^{\alpha}+v_{i+1}^{\alpha}}+\frac{1}{n}\frac{v_i^{\alpha}}{v_i^{\alpha}+v_{i-1}^{\alpha}}.\label{circle_equilibrium}
\end{align}
Moreover,
\begin{align*}
F(\vec{v})_i=-v_i+\frac{1}{n}\frac{v_i^{\alpha}}{v_i^{\alpha}+v_{i+1}^{\alpha}}+\frac{1}{n}\frac{v_i^{\alpha}}{v_i^{\alpha}+v_{i-1}^{\alpha}}
\end{align*}
has derivatives
\begin{align*}
D_{i,i}(\vec{v})=&-1+\frac{\alpha v_i^{\alpha-1}}{n}\left[\frac{1}{v_i^{\alpha}+v_{i+1}^{\alpha}}+\frac{1}{v_i^{\alpha}+v_{i-1}^{\alpha}}-\frac{v_i^{\alpha}}{(v_i^{\alpha}+v_{i+1}^{\alpha})^2}+\frac{v_i^{\alpha}}{(v_i^{\alpha}+v_{i-1}^{\alpha})^2}\right],\\
D_{i,i+1}(\vec{v})=&-\frac{\alpha v_{i+1}^{\alpha-1}}{n}\frac{v_i^{\alpha}}{(v_i^{\alpha}+v_{i+1}^{\alpha})^2},\\
D_{i,i-1}(\vec{v})=&-\frac{\alpha v_{i-1}^{\alpha-1}}{n}\frac{v_i^{\alpha}}{(v_i^{\alpha}+v_{i-1}^{\alpha})^2},\\
D_{i,k}(\vec{v})=&~0 \quad\text{ otherwise.}
\end{align*}
For $\vec{v}=\vec{1}/n$, these reduce to 
\begin{align*}
D_{i,i}(\vec{1}/n)=&-1+\frac{\alpha}{2},\\
D_{i,i+1}(\vec{1}/n)=&D_{i,i-1}(\vec{1}/n)=-\frac{\alpha}{4},\\ 
D_{i,k}(\vec{1}/n)=&~0 \quad \text{ otherwise.}
\end{align*}
Thus, ${\bf D}$ is a circulant matrix with 3 consecutive $(\!\!\!\!\mod (n-1))$ non-zero entries $-\alpha/4$, $-1+\alpha/2$, $-\alpha/4$.  Therefore its eigenvalues are of the form 
\begin{align*}
\lambda_j=&-1+\frac{\alpha}{2}-\frac{\alpha}{4}\e^{2\pi i j/n} -\frac{\alpha}{4}\e^{-2\pi i j/n}\\
=&-1+\frac{\alpha}{2}-\frac{\alpha}{2}\cos(2\pi j/n),
\end{align*}
for $j=0,\dots, n-1$.
All of these eigenvalues are negative if and only if for every $j=0,\dots,n-1$,
\begin{align*}
\alpha [1-\cos(2\pi j/n)]<2.
\end{align*}
When $n$ is even, the left hand side attains its maximum of $2\alpha$ at $j=n/2$ for which the stability criterion is $\alpha<1$.  When $n$ is odd, the left hand side attains its maximum at $j=(n+1)/2$ for which the stability criterion becomes 
\begin{align*}
\alpha<\frac{2}{1-\cos(\pi(1+1/n))}=\frac{2}{1+\cos(\pi/n)}=\frac{1}{\cos\left(\pi/2n\right)^2},
\end{align*}
where the right hand side is greater than 1.  
\hfill\Qed

Note that for $n=3$, this reduces to $\alpha<2/(1+1/2)=4/3$, which must be the case since for $n=3$ this corresponds to the case of fixed $m=2$, uniform $A_t$ (with $n=3$).  By Theorem \ref{thm:subgraphs}, for $n$ even, the vector $\vec{v}_{\text{alt}}=2(1,0,1,0,\dots,1,0)/n$ is a linearly-stable equilibrium for all $\alpha>1$.
%Then \eqref{circle_equilibrium} becomes
%\begin{align} 
%v_{2i-1}=&\frac{1}{n}+\frac{1}{n}\\
%v_{2i}=&0+0,
%\end{align}
%so that $\vec{v}_{\text{alt}}$ is an equilibrium distribution for all $\alpha$.  

\section{Discussion and open problems}
\label{sec:discussion}
\paragraph{Regarding Conjecture \ref{con:whisker-forest}.}
We have shown that when $G$ is the triangle graph and $\alpha>4/3$, any stable equilibrium has some $v_i=0$.  We believe that the same is true (for $\alpha>\alpha_{\sss G}$) when $G$ is the line graph on 4 edges.  Assuming that this can be verified, it is reasonable to expect that for any fixed $G$, and all $\alpha$ sufficiently large, the only linearly-stable equilibria are those admitted by whisker-forests.

We have shown that for all $\alpha>\alpha(r)$ there is a linearly-stable equilibrium (or a unique equilibrium that is critical) on a symmetric whisker-graph.  We expect that the symmetry property is not needed.  If this can be verified, it would imply that for any $G$, any whisker-forest admits a stable equilibrium for $\alpha$ sufficiently large.  There are a great many problems about WARMs that remain open, among them are the following:

\begin{itemize}
\item[(i)]
Is it true that all $\vec{v}\in \Ea$ for a WARM line graph with $3$ edges are symmetric (i.e.,~that $v_1=v_3$)?
\item[(ii)]
Can one prove non-convergence to linearly-unstable equilibria in our general setting?
\item[(iii)] Is it in our general setting true that $\mc{A}\subset \Sa$ when $\Sa\ne \varnothing$?
\end{itemize}

\paragraph{More general models.} This work is inspired by modelling of the brain. We think of the signal entering, giving rise to our generalized P\'olya 
urn. However, in the brain, signals are transmitted between several neurons, suggesting a model where signals perform a random motion (with or without branching of the signal). Without branching, this could be modelled using edge-reinforced random walks (see e.g., \cite{Dav90, DurKesLim02, Pema88, Pem07, LimTar07, MerRol09} and the references therein) on graphs, killed at certain vertices. With branching, this would give rise to a certain kind of branching reinforced walk with killing. Such problems have attracted substantial attention oven the past decade.

\section*{Acknowledgements}
MH thanks Florina Halasan for helpful discussions regarding Lemma \ref{lem:MMDL}.  The work of RvdH was supported in part by the Netherlands Organisation for Scientific
Research (NWO). Holmes's research was supported in part by the Marsden Fund, administered by RSNZ. A. Kuznetsov acknowledges the support by the Natural Sciences and Engineering Research Council of Canada.

\begin{appendices}
\numberwithin{equation}{section}
\section{- $\quad $Proof of Theorem \ref{thm:convergence}}
\label{sec:appendix}

The proof of Theorem \ref{thm:convergence} follows the proof of \cite[Theorem 1.2]{BBCL12} very closely. We repeat this argument almost exactly, only modifying the expression of the Lyapunov function and some related objects. We have included this material for the sake of completeness. 

The main idea of the proof of Theorem \ref{thm:convergence} is to interpret the evolution of the WARM as a stochastic approximation algorithm (see \cite{Benaim1999}). We introduce several definitions and notations. We recall that 
 $N_t^{\sss (i)}$ denotes the number of balls of colour $i$ at time $t\in \Z^+$, $N_0^{\sss (i)}=1$ and $n$ is the total number of colours.
We assume that $p_{\varnothing}=0$, therefore the total number of balls at time $t$ is $n+t$. 
We denote $X_t^{\sss (i)}:=N_t^{\sss (i)}/(n+t)$ to be the proportion of balls of colour $i$.
We define $C_t^{\sss (i)}$ be the number of balls of colour $i$ which is added to the urn at time $t$, that is
$C_t^{\sss (i)}:=N_{t+1}^{\sss (i)}-N_t^{\sss (i)}$.  
We denote ${\mathcal F}_t:=\sigma\{\vec{N}_u \colon 1\le u  \le t\}$. Note that 
$C_t^{\sss (i)} \in \{0,1\}$ is a Bernoulli random variable, such that
\beq\label{prob_Cti}
{\mathbb P}(C_t^{\sss (i)}=1\mid {\mathcal F}_t)=\sum\limits_{A \colon i\in A} p_A \frac{(X_t^{\sss (i)})^{\alpha}}{\sum\limits_{j \in A} (X_t^{\sss (j)})^{\alpha}},
\eeq
moreover, we have $\sum_{i=1}^n C_t^{\sss (i)}=1$ (since only one ball is added to the urn at time $t$). 
By definition, we have $N_{t+1}^{\sss (i)}=N_t^{\sss (i)}+C_t^{\sss (i)}$, therefore
\beq\label{convergence_eqn_1}
X_{t+1}^{\sss (i)}-X_t^{\sss (i)}=\frac{1}{n+t+1} \left( -X_t^{\sss (i)}+C_{t}^{\sss (i)} \right).
\eeq
Denoting
\beqq
F_i(x_1,x_2,\dots,x_n):=-x_i+\sum\limits_{A \colon i\in A} p_A \frac{x_i^{\alpha}}{\sum\limits_{j \in A} x_j^{\alpha}},
\eeqq
and using \eqref{prob_Cti}, we can rewrite \eqref{convergence_eqn_1} in the form
\beq\label{convergence_eqn}
\vec{X}_{t+1}-\vec{X}_t=\gamma_t ( F(\vec{X}_t)+\vec{u}_t),
\eeq
where $F=(F_1,F_2,\dots,F_n)$, $\gamma_t:=1/(n+t+1)$ and $u_t^{\sss (i)}:=C^{\sss (i)}_t-{\mathbb E}[C^{\sss (i)}_t \mid {\mathcal F}_t]$. 
Formula \eqref{convergence_eqn} expresses the WARM as a stochastic approximation algorithm. This is a classical approach to studying convergence of generalized Polya urns, as there exists a well-developed theory for stochastic approximation algorithms (see
\cite{Benaim1999,chen,Kushner_Yin}). In particular, the result of Theorem \ref{thm:convergence}, (ii) follows at once from 
\eqref{convergence_eqn} and \cite[Proposition 7.5]{Benaim1999}.

We write $A\sqsubset [n]$ when $A\subset [n]$ and $p_A>0$.  Let us denote $c:=\tfrac{1}{2}\min \{ p_A \colon A\sqsubset [n]\}$. 
%\RvdH{What does $\sqsubset$ mean?}
We define $\Delta$ to be the set of $n$-tuples $(x_1,x_2,\dots,x_n) \in \R^n$ such that 
\begin{enumerate}
\item $x_i \ge 0$ and $\sum_{i=1}^n x_i=1$, and
\item for all $A\sqsubset [n]$ we have $\sum_{i \in A} x_i \ge c$.
\end{enumerate}
Clearly $F: \Delta \mapsto T\Delta$ is Lipschitz. The following lemma is an analogue of \cite[Lemma 3.4]{BBCL12}:
\begin{LEM}\label{lem_positive_invariance}
$\Delta$ is positively invariant under the ODE  $\;\frac{\d \vec{v}(t)}{\d t}=F(\vec{v}(t))$. 
\end{LEM}
\begin{proof}
If $v$ belongs to the boundary of $\Delta$, then either $v_i=0$ for some $i \in [n]$, or there exists a set $A\sqsubset [n]$
with $\sum_{i \in A} v_i =c$. In the former case, since $F_i(\vec{v})=0$ if $v_i=0$, it is clear that $v(t)$ will stay on the corresponding boundary. Let us consider the latter case. 
Given a set $A$ with $p_A>0$, we have
\beqq
\frac{\d}{\d t} \sum\limits_{i\in A} v_i=\sum\limits_{i \in A} \Big(-v_i+\sum\limits_{B \colon i \in B} p_B
\frac{v_i^{\alpha}}{\sum\limits_{j\in B} v_j^{\alpha}} \Big) \ge 
\sum\limits_{i \in A} \Big(-v_i+ p_A
\frac{v_i^{\alpha}}{\sum\limits_{j\in A} v_j^{\alpha}} \Big)=-\sum\limits_{i \in A} v_i+p_A.
\eeqq
If $v$ is on the boundary of $\Delta$ and there exists a set $A$ such that $\sum_{i \in A} v_i =c$, then
\beqq
\frac{\d}{\d t} \sum\limits_{i\in A} v_i \ge -\sum\limits_{i \in A} v_i+p_A=-c+p_A>0,
\eeqq
which means that $F$ points inward on the boundary of $\Delta$. 
\end{proof}

We recall that ${\mathcal E}={\mathcal E}_{\alpha}$ denotes the set of equilibria of the WARM (the set of solutions to 
$F(\vec{v})=\vec{0}$). 
\begin{DEF}[Strict Lyapunov function] A strict Lyapunov function for a vector field $F$ is a continuous map $L: \Delta \mapsto \R$ which is strictly monotone along any integral curve of $F$ outside of ${\mathcal E}$. In this case, we call $F$ gradient-like.
\end{DEF}

We define a function $L \colon \Delta \mapsto \R$ as
\beq\label{def_Lyap}
L(x_1,x_2,\dots,x_n)=-\sum\limits_{i=1}^n x_i + \frac{1}{\alpha} \sum\limits_{A} p_A \ln \Big(\sum\limits_{j\in A} x_j^{\alpha} \Big).
\eeq
One can check that 
\beq\label{L_main_property}
x_i \frac{\partial L}{\partial x_i}=-x_i+\sum\limits_{A : i \in A} p_A \frac{x_i^{\alpha}}{\sum\limits_{j\in A} x_j^{\alpha}}=F_i(\vec{x})
\eeq
The following result is an analogue of \cite[Lemma 4.1]{BBCL12}:
\begin{LEM}
$L$ is a strict Lyapunov function for $F$. 
\end{LEM}

\begin{proof} Assume that $v(t)=(v_1(t),v_2(t),\dots,v_n(t))$ is an integral curve of $F$, which means that 
$\frac{\d v}{\d t}=F(v)$, then
\beqq
\frac{\d}{\d t} L(v(t))=\sum\limits_{i=1}^n \frac{\partial L}{\partial x_i} \frac{\d v_i}{\d t}=\sum
\limits_{i=1}^n v_i \left(\frac{\partial L}{\partial x_i}\right)^2 \ge 0.
\eeqq
The last expression is zero if and only if $v_i \left(\frac{\partial L}{\partial x_i}\right)^2=0$ for all $i \in[n]$, which is equivalent to $F(v)=0$ (or $v \in {\mathcal E}$). 

\end{proof}

The proof of Theorem \ref{thm:convergence}(i) relies on the following result (see \cite{Benaim1996}, \cite{Benaim1999}
and \cite[Theorem 3.3]{BBCL12}):
\begin{THM}\label{thm_Benaim}
Let $F\colon \R^n\mapsto \R^n$ be a continuous gradient-like vector field with unique integral curves, let ${\mathcal E}$ be its equilibria set, let
$L$ be a strict Lyapunov function, and let $\vec{X}_t$ be a solution to the recursion
\eqref{convergence_eqn}, 
where $(\gamma_t)_{t\ge 0}$ is a decreasing  sequence and $(\vec{u}_t)_{t\ge 0} \subset \R^n$. Assume that 
\begin{itemize}
\item[(i)] $({\vec{X}}_t)_{t\ge 0}$ is bounded,
\item[(ii)] for each $T>0$, 
\beqq
\lim\limits_{t\to +\infty} \Big( \sup\limits_{\{k \colon 0\le \tau_k-\tau_n \le T\}} \Big\| \sum\limits_{i=n}^{k-1} \gamma_i \vec{u}_i  \Big\|\Big)=0,
\eeqq
where $\tau_n=\sum\limits_{i=0}^{n-1} \gamma_i$, and
\item[(iii)] $L({\mathcal E}) \subset \R$ has empty interior.
\end{itemize}
Then the limit set of $(\vec{X}_t)_{t\ge 0}$ is a connected subset of ${\mathcal E}$.
\end{THM}

\noindent
{\bf Proof of Theorem \ref{thm:convergence}(i).}
Again, the proof follows the proof of \cite[Theorem 1.2]{BBCL12} very closely.
Note that $\gamma_t=1/(n+t+1)$ satisfies 
\beqq
\lim\limits_{t\to +\infty} \gamma_t=0,\qquad {\textnormal{ and }} \qquad \sum\limits_{t\ge 0} \gamma_t=+\infty.
\eeqq
It is obvious from the definition that $(\vec{X}_t)_{t\ge 0}$ is bounded, thus 
condition (i)  of Theorem \ref{thm_Benaim} is satisfied. Let us verify condition (ii). We define
\beqq
\vec{M}_t:=\sum\limits_{s=0}^t \gamma_s \vec{u}_s.
\eeqq
It is clear that $(\vec{M}_t)_{t\ge 0}$ is a martingale adapted to the filtration $({\mathcal F}_t)_{t\ge 0}$. Furthermore, 
since for any $t\ge 0$, 
\beqq
\sum\limits_{s=0}^t {\mathbb E}\left[ \| \vec{M}_{s+1}-\vec{M}_s \|^2 | {\mathcal F}_s\right] \le \sum\limits_{s=0}^t \gamma_{s+1}^2 \le \sum\limits_{s=0}^{\infty} \gamma_{s}^2 < \infty,
\eeqq
the sequence $(\vec{M}_t)_{t\ge 0}$ converges almost surely and in $L^2$ to a finite random vector. In particular, it is a Cauchy sequence and therefore, the condition (ii) holds almost surely. 

Now we need to verify condition (iii) in Theorem \ref{thm_Benaim}. We need to distinguish between equilibria lying in the interior of 
${\mathcal E}$ and those lying on the boundary. For each subset $S\subset [n]$, we define
\beqq
\Delta_S:=\{ v\in \Delta \colon v_i=0 {\textnormal{ iff }} i \notin S\}. 
\eeqq
We see that $\Delta_S$ is a face of $\Delta$, it is also a manifold with corners, and, extending the result of Lemma \ref{lem_positive_invariance}, it is easy to see that $\Delta_S$ is positively invariant under the ODE
$\frac{\d \vec{v}}{\d t}=F(\vec{v})$. 

\begin{DEF}
$\vec{v}\in \Delta_S$ is an $S$-singularity for $L$ if 
\beqq
\frac{\partial L}{\partial v_i}(\vec{v})=0 \; {\textnormal{ for all }} i \in S. 
\eeqq
\end{DEF}
Let ${\mathcal E}_{\sss S} \subset \Delta_S$ denote the set of $S$-singularities for $L$. 

\begin{LEM}
${\mathcal E} = \cup_{S\subset [n]} {\mathcal E}_{\sss S}$.
\end{LEM}
\begin{proof}
$\vec{v} \in {\mathcal E}$ means that $F(\vec{v})=0$, and due to \eqref{L_main_property} this is equivalent to 
$v_i \frac{\partial L}{\partial v_i}=0$. Therefore, $\vec{v} \in {\mathcal E}$ implies that for all $i \in [n]$,  either $v_i=0$ or $\frac{\partial L}{\partial v_i}=0$. 
\end{proof}

In order to check condition (iii) of Theorem \ref{thm_Benaim}, we need to show that $L({\mathcal E})=0$. 
For any $S\subset [n]$, the function $L$ restricted to $\Delta_S$ is a $C^{\infty}$ function, thus by Sard's theorem 
$L({\mathcal E}_{\sss S})$ has zero Lebesgue measure, which implies that $L({\mathcal E})$ has zero Lebesgue measure, which in turn implies that 
$L({\mathcal E})$ has empty interior. This verifies condition (iii) in Theorem \ref{thm_Benaim}, and ends the proof of 
Theorem \ref{thm:convergence}(i).
\qed

\end{appendices}

%%%%%%%%%%%%%%%%%%%%%%%%%%%%%%	
\bibliographystyle{plain}

\end{document}